\documentclass[11pt]{amsart}

\usepackage[english]{babel}
\usepackage[utf8x]{inputenc}
\usepackage{amsmath,amsthm,amsfonts,amssymb}
\usepackage{graphicx}
\usepackage[colorinlistoftodos]{todonotes}   
\usepackage{hyperref}
\usepackage{tikz}
\usetikzlibrary{arrows}

\newtheorem{theorem}{Theorem}[section]
\newtheorem{corollary}[theorem]{Corollary}
\newtheorem{lemma}[theorem]{Lemma}
\newtheorem{proposition}[theorem]{Proposition}

\theoremstyle{definition}
\newtheorem{definition}[theorem]{Definition}
\newtheorem{example}[theorem]{Example}
\newtheorem{problem}{Problem}

\newtheorem{remark}[theorem]{Remark}
\numberwithin{equation}{section}

\begin{document}
\title{Connectedness of two-sided group digraphs and graphs}
\author{Patreck Chikwanda}
\address{Department of Mathematics and Statistics,
Georgia State University,
Atlanta, GA 30303}
\email{pchikwanda@gsu.edu}
\author{Cathy Kriloff}
\address{Department of Mathematics and Statistics, Idaho State University, Pocatello, ID 83209-8085}
\email{krilcath@isu.edu, leeyunt@isu.edu, sandtayl@isu.edu, smitgarr@isu.edu, yerodmyt@isu.edu}
\author{Yun Teck Lee\textsuperscript{*}}
\author{Taylor Sandow\textsuperscript{*}}
\author{Garrett Smith\textsuperscript{*}}
\author{Dmytro Yeroshkin}
\renewcommand{\shortauthors}{P.~Chikwanda, C.~Kriloff, Y.T.~Lee, T.~Sandow, G.~Smith, D.~Yeroshkin}
\thanks{\textsuperscript{*} Supported by Idaho State University Career Path Internship funds.}



\begin{abstract}
Two-sided group digraphs and graphs, introduced by Iradmusa and Praeger, provide a generalization of Cayley digraphs and graphs in which arcs are determined by left and right multiplying by elements of two subsets of the group.
We characterize when two-sided group digraphs and graphs are weakly and strongly connected and count connected components, using both an explicit elementary perspective and group actions.  
Our results and examples address four open problems posed by Iradmusa and Praeger that concern connectedness and valency.  We pose five new open problems.
\end{abstract}
\maketitle

%%%%%%%%%%%%%%%%%%%%%%%%%%%%%%%%%%%%%%%%%%
%              Introduction              %
%%%%%%%%%%%%%%%%%%%%%%%%%%%%%%%%%%%%%%%%%%

\section{Introduction}
\label{sec:intro}

Two-sided group digraphs were introduced as a generalization of Cayley digraphs by Iradmusa and Praeger~\cite{IrdamusaPraeger2016} and independently in~\cite{Anil2012} (see~\cite[Remark 1.6]{IrdamusaPraeger2016}).  Given a group $G$ and a subset $S$ of $G$, the \textit{Cayley digraph} $\text{Cay}(G,S)$ has the elements of $G$ as vertices and a directed arc from $g$ to $h$ when $gh^{-1} \in S$.
Several authors have generalized this idea by relaxing the group conditions or the nature of the multiplication (see~\cite{ABR1990,MSS1992,Gauyacq1997,KP2003}).  The \textit{two-sided group digraph} $2\mathrm{S}(G;L,R)$ also has elements of a group $G$ as vertices, but two nonempty subsets, $L$ and $R$, of $G$ are used to define an arc from vertex $g$ to vertex $h$ in $G$ when $h=l^{-1}gr$ for some $l\in L$ and $r\in R$.  
As with Cayley digraphs, by definition $2\mathrm{S}(G;L,R)$ does not have multiple arcs between two vertices, even though it is possible that $l_1^{-1}gr_1=l_2^{-1}gr_2$ for $l_1\neq l_2$ and $r_1 \neq r_2$ (see Section~\ref{sec:prelim}).
A Cayley digraph is undirected when $S=S^{-1}$ and the digraph $2\mathrm{S}(G;L,R)$ is undirected when $L^{-1}gR=LgR^{-1}$ for all $g \in G$, but we do not assume this.

It is worth noting that a continuous version of a two-sided group digraph has previously appeared in the context of Riemannian geometry as the study of biquotients. Introduced in 1974 by Gromoll and Meyer~\cite{Gromoll-Meyer}, biquotients are viewed as the quotient space of a two-sided Lie group action and have been studied systematically as a source of manifolds with positive and non-negative curvature since the work of Eschenburg~\cite{Eschenburg-Paper, Eschenburg-Habilitation}. We refer to DeVito's thesis~\cite{DeVito-thesis} for a broader overview of the topic.

Iradmusa and Praeger explore several properties of two-sided group digraphs and pose eight open problems.  Here we address the first four problems, which concern valency and connectedness.  It would also be of interest to know whether there exist vertex-transitive two-sided group digraphs that are not isomorphic to Cayley digraphs since these would have potential applications to routing and communication schemes in interconnection networks. 
Indeed, the remaining unresolved questions in~\cite{IrdamusaPraeger2016} primarily address understanding when two-sided group digraphs are vertex-transitive and when they are isomorphic to Cayley digraphs.  In addition, we propose five new problems related to our results below.

Our main focus is to generalize~\cite[Theorem~1.8]{IrdamusaPraeger2016}, which gives necessary and sufficient conditions
for a two-sided group digraph $2\mathrm{S}(G;L,R)$ to be connected, assuming that $L$ and $R$ are inverse-closed. 
Theorem~\ref{theorem:main_connectedness_result} solves Problem~4 in~\cite{IrdamusaPraeger2016} by characterizing when $2\mathrm{S}(G;L,R)$ is connected without the inverse-closed assumption on $L$ and $R$.  Examples~\ref{example:A4a} through~\ref{example:A4b} in Section~\ref{sec:prelim} both illustrate Theorem~\ref{theorem:main_connectedness_result} and address Problems~1 and~2 in~\cite{IrdamusaPraeger2016} by showing that it is possible for $2\mathrm{S}(G;L,R)$ to have constant out-valency but not constant in-valency and to be regular of valency strictly less than $|L|\cdot|R|$.

In Section~\ref{sec:connected_components}, building on results in Section~\ref{sec:connection_length}, we use elementary methods similar to those in our proof of Theorem~\ref{theorem:main_connectedness_result} to generalize further.  In Theorems~\ref{theorem:number_weak_components} and~\ref{theorem:number_strong_components}, under the assumption that elements in $G$ can be factored appropriately, we count weakly and strongly connected components, show such components must all be of the same size, and characterize their vertices.  The result that all components have the same size addresses Problem~3 of~\cite{IrdamusaPraeger2016}. We also show that the connected components are in fact isomorphic under a condition on the normalizers of $L$ and $R$.  
To illustrate we provide Corollaries~\ref{cor:weakconn} and~\ref{cor:strongconn} that give simple characterizations of weak and strong connectedness and give Example~\ref{example:D10} in which components are isomorphic and Example~\ref{example:D6B} in which they are not.

In Section~\ref{sec:double_cosets} we drop the factorization assumptions and note that connected components are contained within double cosets.
Results analogous to those in Section~\ref{sec:connected_components} apply within a given double coset and examples demonstrate that in different double cosets the sizes of the connected components can differ.

A less explicit but more natural approach to counting strongly connected components is to view the components as orbits under a group action and to use a standard result that counts orbits.  This is done in Section~\ref{sec:orbitcounting}.  

In Section~\ref{sec:reduction} we prove that when $G$ is a semi-direct product, $G=H \rtimes K$, it is possible to determine whether $2\mathrm{S}(G;L,R)$ is connected by analyzing connectedness properties related to $H$ in $2\mathrm{S}(G;L,R)$ and a two-sided group digraph on $K$. 
We also generalize this to the case where $K$ is $G/H$ for $H$ a normal subgroup of $G$.

%%%%%%%%%%%%%%%%%%%%%%%%%%%%%%%%%%%%%%%%%%
%              Preliminaries             %
%%%%%%%%%%%%%%%%%%%%%%%%%%%%%%%%%%%%%%%%%%

\section{Preliminaries}
\label{sec:prelim}

Following some definitions, we begin with an initial result that characterizes when a two-sided group digraph is strongly connected.  After some examples we compare Theorem~\ref{theorem:main_connectedness_result} to~\cite[Theorem~1.8]{IrdamusaPraeger2016}.

Recall the following definition from~\cite{IrdamusaPraeger2016}.

\begin{definition}
For nonempty subsets $L$ and $R$ of a group $G$, a \textit{two-sided group digraph} $2\mathrm{S}(G;L,R)$ has vertex set $G$ and a directed arc $(g,h)$ from $g$ to $h$ if and only if $h=l^{-1}gr$ for some $l\in L$ and $r\in R$.
\end{definition}
The digraph $2\mathrm{S}(G;L,R)$ is undirected when $L^{-1}gR=LgR^{-1}$ for all $g \in G$, but we work in the generality of directed graphs and consider this situation to be a special case.

\begin{definition}  \label{def:word}
Let $S$ be a nonempty subset of a group $G$. A \textit{word in $S$ of (finite) length $n>0$} is a string $s_1s_2 \cdots s_n$ where $s_1,s_2, \dots, s_n \in S$.  In general, we denote a word in $S$ of length $n$ by $w_{S,n}$ and write $\mathcal{W}(S)$ for the set containing all finite length words in $S$. 
\end{definition}

Note that the factors in a word need not be distinct, a single group element will have numerous different representations as a word in $S$, and different words will be denoted by varying subscripts for the set or length on the letter $w$.

\begin{definition} \label{def:strong_connectedness}
If $g$ and $h$ are vertices in a digraph, then $g$ is \textit{strongly connected} to $h$ if there exists a directed path from $g$ to $h$ and a directed path from $h$ to $g$. A digraph is \textit{strongly connected} if every pair of vertices is strongly connected. 
\end{definition}

\begin{theorem} \label{theorem:main_connectedness_result}
The two-sided group digraph $2\mathrm{S}(G;L,R)$ is strongly connected if and only if $G = \mathcal{W}(L^{-1})\mathcal{W}(R) = \mathcal{W}(L)\mathcal{W}(R^{-1})$ and the identity $e = w_{L^{-1},i+1}w_{R,i} = w_{L^{-1},j}w_{R,j+1}$ for some $i,j \in \mathbb{N}$.
\end{theorem}

\begin{proof}
Assume that the two-sided group digraph $2\mathrm{S}(G;L,R)$ is strongly connected. Then given any $g \in G$, there exists a directed path from the identity element $e$ to $g$, meaning $g = w_{L^{-1},n}ew_{R,n} = w_{L^{-1},n}w_{R,n}$.  Hence $g \in \mathcal{W}(L^{-1})\mathcal{W}(R)$ and $G = \mathcal{W}(L^{-1})\mathcal{W}(R)$. Since there also exists a directed path from $g$ to $e$, we know $e = w_{L^{-1},m}gw_{R,m} $ which implies that $g = w_{L^{-1},m}^{-1}w_{R,m}^{-1} = w_{L,m}w_{R^{-1},m}$, and hence $G = \mathcal{W}(L)\mathcal{W}(R^{-1})$.  In particular there exists a directed path from $l^{-1}$ to $e$ where $l \in L$, and hence $e = w_{L^{-1},i} l^{-1} w_{R,i} = w_{L^{-1},i+1}w_{R,i}$ for some $i$. Similarly, $e = w_{L^{-1},j}w_{R,j+1}$ since there is a directed path from $r$ to $e$ where $r \in R$.

Conversely, suppose that $G = \mathcal{W}(L^{-1})\mathcal{W}(R) = \mathcal{W}(L)\mathcal{W}(R^{-1})$ and the identity element $e = w_{L^{-1},i+1}w_{R,i} = w_{L^{-1},j}w_{R,j+1}$ for some $i,j \in \mathbb{N}$. It suffices to show that there is a directed path from $e$ to $g$ and from $g$ to $e$ for all $g \in G$; i.e., $g = w_{L^{-1},m}w_{R,m} = w_{L,n}w_{R^{-1},n}$ for some $m,n \in \mathbb{N}$. 

Since $G = \mathcal{W}(L^{-1})\mathcal{W}(R)$, we know $g$ has an $L^{-1}R$ factorization, i.e., $g = w_{L^{-1},a}w_{R,b}$ for some $a,b \in \mathbb{N}$. If $a \neq b$, it is possible to adjust the $L^{-1}R$ factorization of $g$ so that both words have the same length by inserting the appropriate factorization of $e$ between the words from $L^{-1}$ and $R$. For example, if $a >b$, then insert $e = w_{L^{-1},j}w_{R,j+1}$ to obtain 
\[g = w_{L^{-1},a}w_{R,b} = w_{L^{-1},a}(w_{L^{-1},j}w_{R,j+1})w_{R,b} = w_{L^{-1},a+j}w_{R,b+j+1}.\]
Repeating this process yields $g = w_{L^{-1},m}w_{R,m}$ where $m = a+(a-b)j$.

To see that $g$ also has an $LR^{-1}$ factorization with words of the same length, note that left and right multiplying by inverses of the words from $L^{-1}$ and $R$ respectively converts $e = w_{L^{-1},i+1}w_{R,i} = w_{L^{-1},j}w_{R,j+1}$ into 
$e = w_{L,i+1}w_{R^{-1},i} = w_{L,j}w_{R^{-1},j+1}$. 
Repeatedly inserting the appropriate $LR^{-1}$ factorization of $e$ into an $LR^{-1}$ factorization of $g$ shows $g = w_{L,n}w_{R^{-1},n}$ for any $g \in G$. Hence $2\mathrm{S}(G;L,R)$ is strongly connected.
 \end{proof}

The following examples illustrate Theorem~\ref{theorem:main_connectedness_result} and also address the first two problems posed in\cite{IrdamusaPraeger2016}.

\begin{example}
\label{example:A4a}
Consider $\Gamma = 2\mathrm{S}(A_4;L,R)$ where $A_4$ is the alternating group on four elements, $L=\{e,(243)\}$, and $R=\{(234),(12)(34),(132),(14)(23)\}$,
as shown in Figure~\ref{fig:Problem1}. Since $G$ is generated by words in $R$ or $R^{-1}$, then $G=\mathcal{W}(L^{-1})\mathcal{W}(R)=\mathcal{W}(L)\mathcal{W}(R^{-1})$.  Also $e=e^3\cdot [(12)(34)]^2=e\cdot [(12)(34)]^2$ so the hypotheses of Theorem~\ref{theorem:main_connectedness_result} hold and thus $\Gamma$ is strongly connected.  

This example addresses Problem~1 of~\cite{IrdamusaPraeger2016} which asks whether or not $2\mathrm{S}(G;L,R)$ can have constant out-valency but not constant in-valency. The digraph $\Gamma$ has constant out-valency of $7$, however the vertices $\{(123),\ (132),\ (142),\ (143),\ (12)(34),\ (13)(24)\}$ have in-valency $6$ and the vertices $\{e,\ (234),\ (243),\ (134),\ (124),\ (14)(23)\}$ have in-valency $8$. Furthermore $2\mathrm{S}(A_4;L^{-1},R^{-1})$ will have constant in-valency of $7$ and out-valency of either $6$ or $8$ for the same sets as in $2\mathrm{S}(G;L,R)$ because inverting $L$ and $R$ changes the direction of each edge.
\end{example}

\begin{problem}
\label{problem1}
For $2\mathrm{S}(G;L,R)$ with constant out-valency, what are the possible sets of in-valencies?  In particular, how large can they be and how much can they differ from the out-valency?
\end{problem}

\begin{figure}[htbp]
\centering
\includegraphics[width=6 cm]{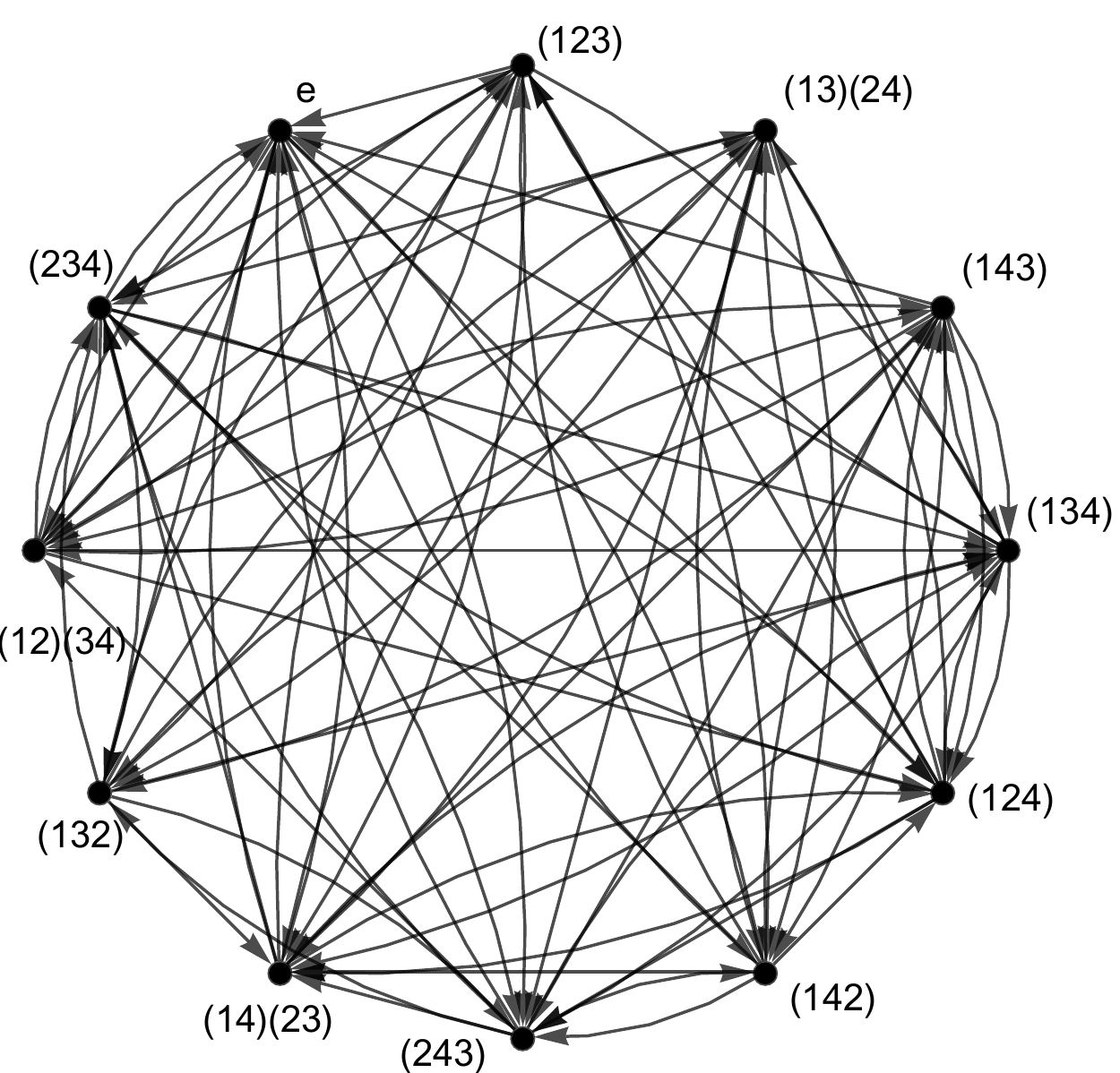}
\caption{
$2\mathrm{S}(A_4;\{e,(243)\},\{(234),(12)(34),(132),(14)(23)\})$}
\label{fig:Problem1}
\end{figure}

\begin{example}
\label{example:C7}
The two-sided group digraph $2\mathrm{S}(C_7;\{g^2,g^3\},\{e,g\})$, where $C_7$ is the cyclic group of order seven generated by $g$, satisfies the hypotheses of Theorem~\ref{theorem:main_connectedness_result} and is connected.
This example also addresses Problem~2 of~\cite{IrdamusaPraeger2016}, which asks whether or not $2\mathrm{S}(G;L,R)$ can be a regular graph of valency strictly less than $|L|\cdot|R|$.  Here $|L|\cdot |R|=4$, but as seen in Figure~\ref{fig:C7Ex}, $2\mathrm{S}(C_7;\{g^2,g^3\},\{e,g\})$ is regular with valency three.  In fact $2\mathrm{S}(C_7;\{g^2,g^3\},\{e,g\}) \cong \text{Cay}(C_7,\{g^4,g^5,g^6\})$ with $g^5$ arising in two different ways from the sets $L^{-1}$ and $R$, explaining the valency of three. 
\end{example}

\begin{figure}[htbp]
\centering
\includegraphics[width=5 cm]{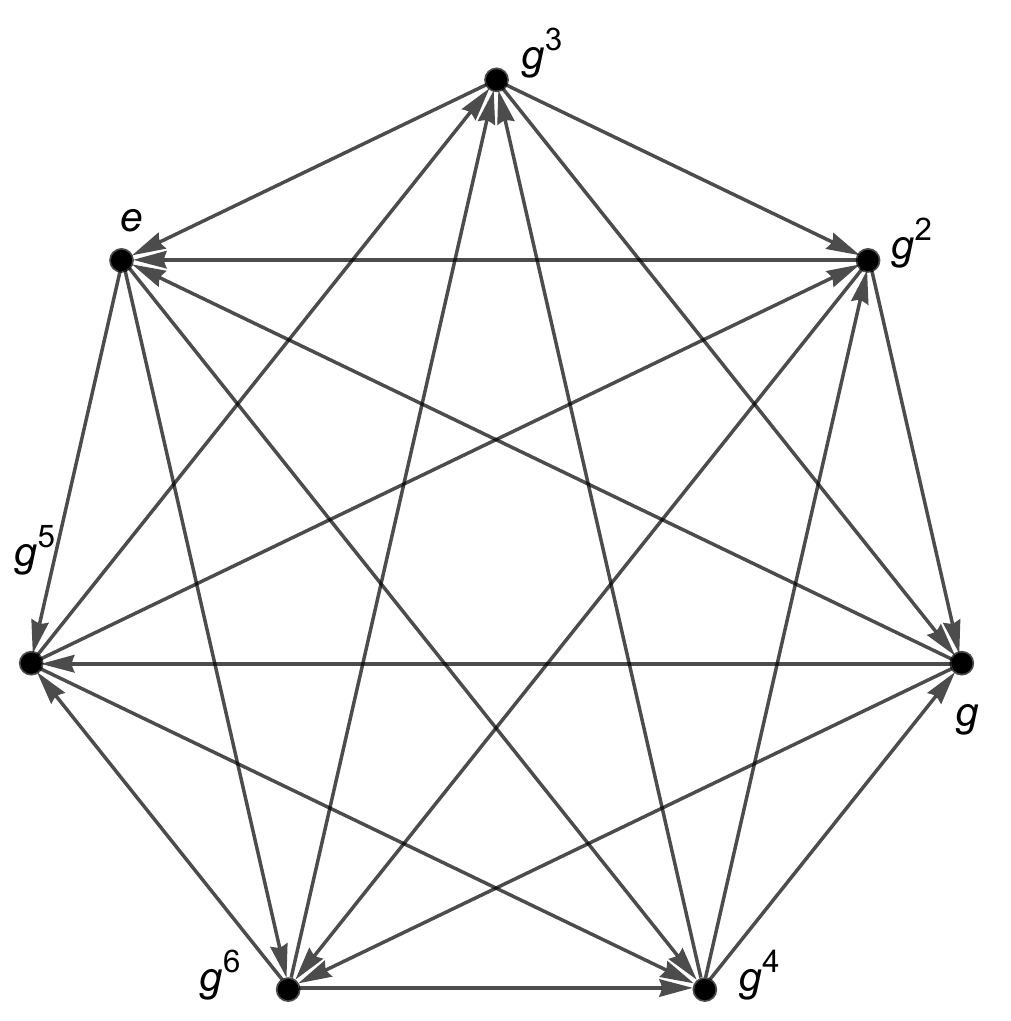}
\caption{
$2\mathrm{S}(C_7;\{g^2,g^3\},\{e,g\})$}
\label{fig:C7Ex}
\end{figure}

\begin{example}
\label{example:D6}
Consider the dihedral group, $D_6$, of order $12$, generated by the reflection $\tau$ and the rotation $\sigma$ of order $6$.  The undirected graph
$2\mathrm{S}(D_6;\{\tau,\tau\sigma^5\},\{\tau\sigma,\tau\sigma^2\})$
is regular of valency three and $|L|\cdot|R|=4$ as in Example~\ref{example:C7}, but for a non-abelian group.
For any $g \in D_6$, the set $(LL^{-1})^g\cap (RR^{-1})=\{e,\sigma,\sigma^{-1}\}$ is of size three and there is a reduction in valency by one.
See Figure~\ref{fig:Problem2D6} in Section~\ref{sec:connected_components}.
\end{example}

\begin{example}
\label{example:A4b}
The two-sided group digraph $2\mathrm{S}(A_4;A_4,\{(243),(12)(34)\})$ has $|L|\cdot |R|=24$, but is in fact regular with valency $12$ and forms the complete undirected graph with loops.  Here $(LL^{-1})^g\cap (RR^{-1})=\{e,(124),(142)\}$ is of size three for each $g$, but the reduction in valency is much larger than in the previous examples because $L^{-1}gR$, viewed as a multiset, consists of $12$ distinct elements, each with multiplicity two.
\end{example}
The set $(LL^{-1})^g\cap (RR^{-1})$, where $(LL^{-1})^g=g^{-1}LL^{-1}g$, is introduced
in~\cite[Definition~1.4(3) and Theorem~1.5]{IrdamusaPraeger2016}, and the condition $(LL^{-1})^g\cap (RR^{-1})=\{e\}$ is shown in~\cite[Lemma~3.1]{IrdamusaPraeger2016} to guarantee the valency of $2\mathrm{S}(G;L,R)$ is exactly $|L|\cdot|R|$.
In Examples~\ref{example:C7},~\ref{example:D6}, and~\ref{example:A4b} it is the failure of this condition which causes a drop in valency.
In general, if for $h \neq e$, we have $h=g^{-1}l_1l_2^{-1}g=r_1r_2^{-1}$ for some $l_1,l_2 \in L$ and some $r_1, r_2 \in R$, then $l_1^{-1}gr_1=l_2^{-1}gr_2$, which causes a multiplicity greater than one in $L^{-1}gR$ considered as a multiset.  Since the elements and their multiplicities in the multiset $L^{-1}gR$ depend on $g$, we did not search for necessary and sufficient conditions on $L$ and $R$ for a two-sided group digraph to have valency strictly less than $|L|\cdot |R|$ as requested in Problem~2 of~\cite{IrdamusaPraeger2016}.  Thus this aspect of~\cite[Problem 2]{IrdamusaPraeger2016} remains unresolved.

 Theorem~\ref{theorem:main_connectedness_result} is a generalization of the first part of the following result in~\cite{IrdamusaPraeger2016} and also addresses~\cite[Problem 4]{IrdamusaPraeger2016}.
\begin{theorem}~\cite[Theorem~1.8]{IrdamusaPraeger2016}
\label{theorem:IPtheorem}
Let $L$ and $R$ be nonempty inverse-closed subsets of a group $G$, and let $\Gamma=2\mathrm{S}(G;L,R)$.  
\begin{enumerate}
\item The graph $\Gamma$ is connected if and only if $G=\langle L \rangle \langle R \rangle$ and there exist words in $L$ and $R$ with length of opposite parity whose product is $e$.
\item If $G=\langle L \rangle \langle R \rangle$ and there do not exist words in $L$ and $R$ with length of opposite parity whose product is $e$ then $\Gamma$ is disconnected with exactly two connected components.
\end{enumerate}
\end{theorem}

Theorems~\ref{theorem:number_weak_components} and~\ref{theorem:number_strong_components} further generalize~\cite[Theorem~1.8]{IrdamusaPraeger2016} by providing more general counts and characterizations of connected components.  Theorem~\ref{theorem:number_weak_components} also answers~\cite[Problem~3]{IrdamusaPraeger2016}, by showing there cannot exist $G$, $L$, and $R$ satisfying the hypotheses of Theorem~\ref{theorem:IPtheorem} such that $G=\langle L \rangle \langle R \rangle$ but $2\mathrm{S}(G;L,R)$ has connected components of different sizes.
We show more generally that if $G=\mathcal{W}(L \cup L^{-1})\mathcal{W}(R \cup R^{-1})$ then all connected components of $2\mathrm{S}(G;L,R)$ have the same size.

%%%%%%%%%%%%%%%%%%%%%%%%%%%%%%%%%%%%%%%%%%
%            Connectedness               %
%%%%%%%%%%%%%%%%%%%%%%%%%%%%%%%%%%%%%%%%%%

\section{General connectedness results}

\subsection{Connection length}\label{sec:connection_length}

In this section we lay the foundation for studying both weakly and strongly connected components of $2\mathrm{S}(G;L,R)$ in Section~\ref{sec:connected_components}.

\begin{definition} \label{def:weak_connectedness}
In a digraph, a vertex $g$ is \textit{weakly connected} to vertex $h$ if there is a path $g_0, g_1, \cdots, g_n$ such that $g=g_0$, $h=g_n$, and  either $(g_{i-1}, g_i)$ or $(g_i, g_{i-1})$ is an arc of the digraph.  A digraph is \textit{weakly connected} if each pair of its vertices is weakly connected. 
\end{definition}

If $L$ and $R$ are nonempty subsets of a group $G$, we let $\bar{L}=L\cup L^{-1}$ and $\bar{R}=R\cup R^{-1}$ and use $w_{\bar{L},m,a}$ to denote a word that contains $m$ factors from $L$ and $a$ factors from $L^{-1}$ in any order.
The notation $g \sim h$ will mean $g$ is weakly connected to $h$ in $2\mathrm{S}(G;L,R)$, or equivalently  $h=W_{\bar{L},m,a}gW_{\bar{R},a,m}$, where the capital $W$ indicates that the corresponding factors on either side of $g$ have opposite signs, i.e., one factor from $L^{-1}$ and one from $R$, or alternatively, one from $L$ and one from $R^{-1}$.  If computations lead to factorizations that may not involve opposite signs on corresponding factors then $W$ is changed to $w$.

We begin with two key results which will allow us to define minimum weak connection length in Definition~\ref{def:minweakconnlengthG} and which will also be used in the proof of Theorem~\ref{theorem:number_weak_components}.

\begin{lemma}
\label{lem:weak_word_manipulation}
In $2\mathrm{S}(G;L,R)$ if $g=w_{\bar{L},m,a}w_{\bar{R},n,b}$, then $g \sim l^d$ and $g \sim r^d$ where $l$ is any element of $L$, $r$ is any element of $R$, and $d=m+n-(a+b)$.
\end{lemma}

\begin{proof}
Let $g=w_{\bar{L},m,a}w_{\bar{R},n,b}$ for $a, b, m, n \in \mathbf{N}$. Then we have for any $r \in R$ and $l \in L$,
\begin{align*}
g &= w_{\bar{L},m,a}w_{\bar{R},n,b} \\
  &= w_{\bar{L},m,a}w_{\bar{R},n,b}r^{m-a}r^{-m+a}\\
  &= W_{\bar{L},m,a}w_{\bar{R},m+n,a+b} W_{\bar{R},a,m} \\
  &= W_{\bar{L},m,a}l^{a+b-(m+n)}l^{m+n-(a+b)}w_{\bar{R},m+n,a+b} W_{\bar{R},a,m} \\
  &=W_{\bar{L},m,a}W_{\bar{L},a+b,m+n}l^{m+n-(a+b)}W_{\bar{R},m+n,a+b}W_{\bar{R},a,m}.
 \end{align*}
Corresponding factors can be adjusted to have opposite signs because the repeated $r$ and $r^{-1}$ and $l$ and $l^{-1}$ can be rearranged as needed. A similar construction will yield $g \sim r^d$. 
\end{proof}

The following corollary is stated in terms of $L$, but an analogous statement in terms of $R$ also holds.

\begin{corollary}
\label{cor:weak_connection_properties}
Let $L$ and $R$ be nonempty subsets of a group $G$.
\begin{enumerate}
\item In $2\mathrm{S}(G;L,R)$ there exist two words in $L$ of different lengths that are weakly connected if and only if there is a word in $L$ that is weakly connected to $e$. 
\item In $2\mathrm{S}(G;L,R)$ there exists a word $w_{L,n}$ weakly connected to $e$ if and only if there exists a word $w_{L^{-1},n}$ weakly connected to $e$.
\end{enumerate}
\end{corollary}

\begin{proof}
For (1), if $w_{L,m} \sim w_{L,n}$, assume without loss of generality that $m<n$, left multiply by $w_{L,m}^{-1}$, and apply Lemma~\ref{lem:weak_word_manipulation} to obtain $e \sim l^{n-m}$ for $l \in L$. Conversely if $e \sim w_{L,m}$, left multiply by some $l \in L$ and apply Lemma~\ref{lem:weak_word_manipulation}. 

For (2), if $e = W_{\bar{L},m,a} w_{L,n} W_{\bar{R},a,m}$ then $e =w_{L,n}^{-1} w_{\bar{L},a,m}w_{\bar{R},m,a}$. Now apply Lemma~\ref{lem:weak_word_manipulation} to obtain $e \sim (l^{-1})^n$ for $l \in L$. The converse is achieved analogously. 
\end{proof}

These results yield that the following notion is well-defined.

\begin{definition}
\label{def:minweakconnlengthG}
The \textit{minimum weak connection length in $G$ relative to $(L,R)$} is the minimum length $k$ of a word purely in $L$, $L^{-1}$, $R$, or $R^{-1}$ that is weakly connected to $e$, and is infinite if there is no such minimum. 
Algebraically this is equivalent to the minimum length of a word $w$ purely in $L$, $L^{-1}$, $R$, or $R^{-1}$ such that $e=W_{\bar{L},m,a}wW_{\bar{R},a,m}$ for some $a,m \in \mathbb{N}$.
\end{definition}

Here and in the next section we impose the additional assumption that the set of words in $L$ and the set of words in $R$ are subgroups of $G$ in order to adapt weak connectedness results to the case of strong connectedness using Proposition~\ref{lem:two_way_connectedness} and Corollary~\ref{cor:remark1}. 
The following proposition provides two further means of verifying that sets of words are subgroups.

\begin{proposition} \label{prop:equiv_hypotheses} 
Given any nonempty subset $S$ of a group $G$ the following are equivalent:
\begin{enumerate}
\item $\mathcal{W}(S)$ is a subgroup of $G$,
\item $\mathcal{W}(S) = \mathcal{W}(S^{-1})$,
\item $\mathcal{W}(S) = \langle S \rangle$.  
\end{enumerate}
\end{proposition}

\begin{proof} 
Clearly $(2)$ implies $(3)$ and $(3)$ implies $(1)$ so it remains to show that $(1)$ implies $(2)$. Assume $\mathcal{W}(S)$ is a subgroup of $G$ and let $w \in \mathcal{W}(S)$. Then $w^{-1}\in \mathcal{W}(S)$ by assumption and $w^{-1}=s_1s_2\cdots s_k$ where $k > 0$. Hence $w=(s_1s_2\cdots s_k)^{-1}=s_k^{-1}s_{k-1}^{-1}\cdots s_1^{-1}\in\mathcal{W}(S^{-1})$.

Now suppose that $w\in\mathcal{W}(S^{-1})$.  Then $w=s_1^{-1}s_2^{-1}\cdots s_k^{-1}=(s_ks_{k-1}\cdots s_1)^{-1}\in\mathcal{W}(S)$ because $\mathcal{W}(S)$ is a subgroup of G. Hence $\mathcal{W}(S)=\mathcal{W}(S^{-1})$ and the result follows. 
\end{proof}

\begin{remark}
Notice that if $G$ is a finite group, any subset $S$ of $G$ will satisfy the statements in Proposition~\ref{prop:equiv_hypotheses}.  The statements will also hold in any group if the subset $S$ is inverse-closed, as is assumed in places in~\cite{IrdamusaPraeger2016}, or if all elements of $S$ have finite order.
\end{remark}

\begin{proposition}
\label{lem:two_way_connectedness} 
In the two-sided group digraph $2\mathrm{S}(G;L,R)$, if $\mathcal{W}(L)$ and $\mathcal{W}(R)$ are subgroups of $G$, there is a directed path from $g$ to $h$ if and only if there is a directed path from $h$ to $g$. 
\end{proposition}

\begin{proof}
Suppose that there is a directed path from $g$ to $h$ in $2\mathrm{S}(G;L,R)$. Then we have $h = w_{L^{-1},n}gw_{R,n}$ for some $n \in \mathbb{N}$ which implies that $g = w_{L^{-1},n}^{-1} h w_{R,n}^{-1}$. 

Since $\mathcal{W}(L^{-1})$ and $\mathcal{W}(R)$ are both subgroups of $G$, $w_{L^{-1},n}^{-1} \in \mathcal{W}(L^{-1})$ and $w_{R,n}^{-1} \in \mathcal{W}(R)$, i.e., inverses can be expressed as words in the original set. It will be sufficient to show that both of the inverses can be expressed as words in their respective sets with the same length.

First suppose that $w_{L^{-1},n}^{-1} = w_{L^{-1},a}$ and $w_{R,n}^{-1} = w_{R,b}$ for some $a,b \in \mathbb{N}$. Then we have $e = w_{L^{-1},n}w_{L^{-1},a}$ and similarly $e = w_{R,n}w_{R,b}$ of total lengths at least one. Using that $e$ is the identity, we have 
\begin{align*}
e &= w_{L^{-1},n}w_{L^{-1},a}(w_{L^{-1},n}w_{L^{-1},a})^{n+b-1} \\
&= w_{L^{-1},n}w_{L^{-1},a}w_{L^{-1},(n+a)(n+b-1)} \\
&= w_{L^{-1},n}w_{L^{-1},a+(n+a)(n+b-1)}. 
\end{align*}
This shows that $w_{L^{-1},n}^{-1}$ can be expressed as a word in $L^{-1}$ of length $(n+a)(n+b)-n$. Similarly, we can express $w_{R,n}^{-1}$ as a word in $R$ of the same length.  Therefore there is a directed path from $h$ to $g$ in $2\mathrm{S}(G;L,R)$.
\end{proof}

\begin{corollary}\label{cor:remark1}
In the two-sided group digraph $2\mathrm{S}(G;L,R)$, if $\mathcal{W}(L)$ and $\mathcal{W}(R)$ are subgroups of $G$, $g \in G$ is weakly connected to $h\in G$ if and only if $g$ is strongly connected to $h$ and hence weakly connected components are identical to strongly connected components. 
\end{corollary}

\begin{proof}
Assume that $g$ is weakly connected to $h$ in $2\mathrm{S}(G;L,R)$.  Then there exists a path $g_0, g_1, \cdots, g_n$ with $g=g_0$ and $h=g_n$ such that either $(g_{i-1}, g_i)$ or $(g_i, g_{i-1})$ is an arc for $1\leq i\leq n$. For every arc of the form $(g_i,g_{i-1})$, apply Proposition~\ref{lem:two_way_connectedness}. This generates a new directed path $g'_0, g'_1, \cdots, g'_m$ with $g=g'_0$ and $h=g'_m$ such that $(g'_{i-1}, g'_i)$ is an arc for $1\leq i\leq m$. Applying Proposition~\ref{lem:two_way_connectedness} again yields that g is strongly connected to h.
\end{proof}

Under the hypothesis that $\mathcal{W}(L)$ and $\mathcal{W}(R)$ are subgroups of $G$, Proposition~\ref{lem:two_way_connectedness} and Corollary~\ref{cor:remark1} allow us to convert any statement about weak connectedness into a corresponding statement about strong connectedness.  This leads to the following results analogous to Lemma~\ref{lem:weak_word_manipulation} and Corollary~\ref{cor:weak_connection_properties} and consequently a well-defined notion of minimum strong connection length.

\begin{lemma}
\label{lem:strong_word_manipulation}
In $2\mathrm{S}(G;L,R)$ if $\mathcal{W}(L)$ and $\mathcal{W}(R)$ are subgroups of $G$ and $g=w_{L^{-1},a}w_{R,n}$, then $g$ is strongly connected to  $l^d$ and to $r^d$ where $l$ is any element of $L$, $r$ is any element of $R$, and $d=n-a$.
\end{lemma}

\begin{corollary}
\label{cor:strong_connection_properties}
Let $\mathcal{W}(L)$ and $\mathcal{W}(R)$ be subgroups of $G$.
\begin{enumerate}
\item In $2\mathrm{S}(G;L,R)$ there exist two words in $L$ of different lengths that are strongly connected if and only if there is a word in $L$ that is strongly connected to $e$. 
\item In $2\mathrm{S}(G;L,R)$ there exists a word $w_{L,n}$ strongly connected to $e$ if and only if there exists a word $w_{L^{-1},n}$ strongly connected to $e$.
\end{enumerate}
\end{corollary}

\begin{definition}
Assuming that $\mathcal{W}(L)$ and $\mathcal{W}(R)$ are subgroups of $G$, the \textit{minimum strong connection length in $G$ relative to $(L,R)$} is the minimum length $k$ of a word purely in $L$, $L^{-1}$, $R$, or $R^{-1}$ that is strongly connected to $e$, and is infinite if there is no such minimum. 
Algebraically this is equivalent to the minimum length of a word $v$ purely in $L$, $L^{-1}$, $R$, or $R^{-1}$ such that $e=w_{L^{-1},n}vw_{R,n}$ for some $n \in \mathbb{N}$.
\end{definition}

Lemma~\ref{lem:strong_word_manipulation} and Corollary~\ref{cor:strong_connection_properties_coset} also lead to the following version of Theorem~\ref{theorem:main_connectedness_result}.

\begin{corollary}
Let $\mathcal{W}(L)$ and $\mathcal{W}(R)$ be subgroups of $G$.  The two-sided group digraph $2\mathrm{S}(G;L,R)$ is strongly connected if and only if $G = \langle L\rangle \langle R\rangle$ and $e = w_{L^{-1},i}w_{R,j}$ where $|i-j| = 1$.
\end{corollary}

%%%%%%%%%%%%%%%%%%%%%%%%%%%%%%%%%%%%%%%%%%
%               Components               %
%%%%%%%%%%%%%%%%%%%%%%%%%%%%%%%%%%%%%%%%%%

\subsection{Connected components}
\label{sec:connected_components}

In this section we count numbers of connected components and characterize their vertices, assuming that elements of $G$ factor as a word in $\bar{L}=L\cup L^{-1}$ times a word in $\bar{R}=R\cup R^{-1}$.

\begin{theorem}
\label{theorem:number_weak_components}
Let $L$ and $R$ be nonempty subsets of a group $G$. If $G = \mathcal{W}(\bar{L})\mathcal{W}(\bar{R})$ and $k$ is the minimum weak connection length for $G$ relative to $(L,R)$, then the two-sided group digraph $2\mathrm{S}(G;L,R)$ has exactly $k$ weakly connected components all of the same size. Moreover, if $L \cap N_G(L) \neq \emptyset$ or $R \cap N_G(R) \neq \emptyset$, then all components are isomorphic.
\end{theorem}

\begin{proof}
Assume $G = \mathcal{W}(\bar{L})\mathcal{W}(\bar{R})$ and let $k$ be the minimum weak connection length for $G$ relative to $(L,R)$. If $k$ is infinite, then by Corollary~\ref{cor:weak_connection_properties}, any two words in $L$ of different lengths are not weakly connected to each other and it follows that $2\mathrm{S}(G;L,R)$ will have infinitely many connected components. Otherwise $k\in \mathbb{N}$ and by Corollary~\ref{cor:weak_connection_properties}, we may assume $e = W_{\bar{L},m,a} l^k W_{\bar{R},a,m}$.  For $0\leq i<j < k$ we claim that $l^i \neq l^j$ and there is no path between $l^i$ and $l^j$.

If $l^i = l^j$ for some $0 \leq i<j< k$, then $e = l^{j-i}$, contradicting the minimality of $k$ as the weak connection length for $G$ relative to $(L,R)$. Similarly, if $l^i=W_{\bar{L},m,a} l^j W_{\bar{R},a,m}$ then $e = l^{-i}W_{\bar{L},m,a} l^j W_{\bar{R},a,m}$ and Lemma~\ref{lem:weak_word_manipulation} yields $e \sim l^{j-i}$ which again contradicts the minimality of $k$. This shows that $2\mathrm{S}(G;L,R)$ has at least $k$ weakly connected components.

To show that $2\mathrm{S}(G;L,R)$ has exactly $k$ weakly connected components, we first notice that since $G = \mathcal{W}(\bar{L})\mathcal{W}(\bar{R})$, Lemma~\ref{lem:weak_word_manipulation} means that for every $g \in G$, $g \sim l^d$ for some integer $d$. Hence it suffices to show that for all $d \in \mathbb{Z}$, $l^d \sim l^i$ for some $0\leq i < k$. This statement is true since by Lemma~\ref{lem:weak_word_manipulation}, $e \sim l^{-k}$ and $e \sim l^k$ which allow $d$ to be reduced modulo $k$.

Fix $l\in L$ and let $\Gamma_i$ for $0 \leq i < k$ be the weakly connected component of $2\mathrm{S}(G;L,R)$ containing $l^i$. Then the $\Gamma_i$ are distinct and the union of $\Gamma_0,\dots,\Gamma_{k-1}$ is $2\mathrm{S}(G;L,R)$. To see that all of the connected components have the same size, consider the injective maps $\phi_i : \Gamma_0 \rightarrow \Gamma_i$ for $1 \leq i < k$ defined by $\phi_i(h) = l^i h$. The map sending $h$ to $l^{-i} h$ is also injective and is an inverse to $\phi_i$, showing that $\phi_i$ is bijective and all connected components have the same size.

Now assume $l \in L \cap N_G(L)$ and let $\Gamma_i$ and $\phi_i$ be defined as above. The maps $\phi_i$ will preserve arcs because if $(x,y)$ is an arc in $\Gamma_0$ then $y = l_1^{-1}xr_1$ for some $l_1 \in L$ and $r_1 \in R$, and, since $l \in N_G(L)$, for some $l_2 \in L$
\[\phi_i(y) = l^i y = l^i l_1^{-1}xr_1 = l_2^{-1} l^i xr_1 = l_2^{-1}\phi_i(x)r_1.\] Similarly, if $(\phi_i(x),\phi_i(y))$ is an arc in $\Gamma_i$, then $(x,y)$ is an arc in $\Gamma_0$. Thus the disjoint connected components are isomorphic to each other.    

For the case when $R\cap N_G(R)\neq \emptyset$, note that the above proof can be modified using the set $\{r^i\}_{i=0}^{k-1}$ to describe the $\Gamma_i$ and defining $\phi_i(h)=hr^i$ instead. 
\end{proof}

\begin{remark}
Note that in Theorem~\ref{theorem:number_weak_components} if in fact $L\cap N_G(L)=L$ and $R\cap N_G(R)=R$, then  $2\mathrm{S}(G;L,R)$ is also vertex-transitive by~\cite[Theorem~1.13]{IrdamusaPraeger2016}.
\end{remark}

\begin{corollary}
\label{cor:weakconn}
The two-sided group digraph $2\mathrm{S}(G;L,R)$ is weakly connected if and only if $G = \mathcal{W}(\bar{L})\mathcal{W}(\bar{R})$ and there exists some element of $\bar{L}$ or $\bar{R}$ that is weakly connected to $e$. 
\end{corollary}

Using Proposition~\ref{lem:two_way_connectedness} and Corollary~\ref{cor:remark1} as described before Lemma~\ref{lem:strong_word_manipulation} yields the following.

\begin{theorem}
\label{theorem:number_strong_components}
Let $\mathcal{W}(L)$ and $\mathcal{W}(R)$ be subgroups of $G$. If $G = \mathcal{W}(L)\mathcal{W}(R)$ and $k$ is the minimum strong connection length for $G$ relative to $(L,R)$, then the two-sided group digraph $2\mathrm{S}(G;L,R)$ has exactly $k$ strongly connected components all of the same size. Moreover, if $L \cap N_G(L) \neq \emptyset$ or $R \cap N_G(R) \neq \emptyset$, then all components are isomorphic.
\end{theorem}

\begin{corollary}
\label{cor:strongconn}
Let $\mathcal{W}(L)$ and $\mathcal{W}(R)$ be subgroups of $G$. Then the two-sided group digraph $2\mathrm{S}(G;L,R)$ is strongly connected if and only if $G = \mathcal{W}(L)\mathcal{W}(R)$ and there exists some element of $\bar{L}$ or $\bar{R}$ that is strongly connected to $e$.
\end{corollary}

\begin{example}
\label{example:D10}
Consider $2\mathrm{S}(D_6;\{\tau,\tau\sigma^5\},\{\tau\sigma,\tau\sigma^2\})$ as in Example~\ref{example:D6}.
Since $\tau \in L$ and $\sigma =(\tau \sigma^5)\tau \in \mathcal{W}(L)$, we know $D_6=\mathcal{W}(L)=\mathcal{W}(L)\mathcal{W}(R)$.  Since $e \not\sim \tau$ but $e\sim \tau^2=e$, the graph, as seen in Figure~\ref{fig:Problem2D6}, has two strongly connected components of the same size as shown in Theorem~\ref{theorem:number_strong_components}.  Notice that $N_{D_6}(L)=N_{D_6}(R)=\{e,\sigma^3\}$ does not intersect $L$ or $R$ so the fact that the components are not isomorphic does not violate Theorem~\ref{theorem:number_strong_components}.
\end{example}

\begin{figure}[htbp]
\centering
\includegraphics[width=10cm]{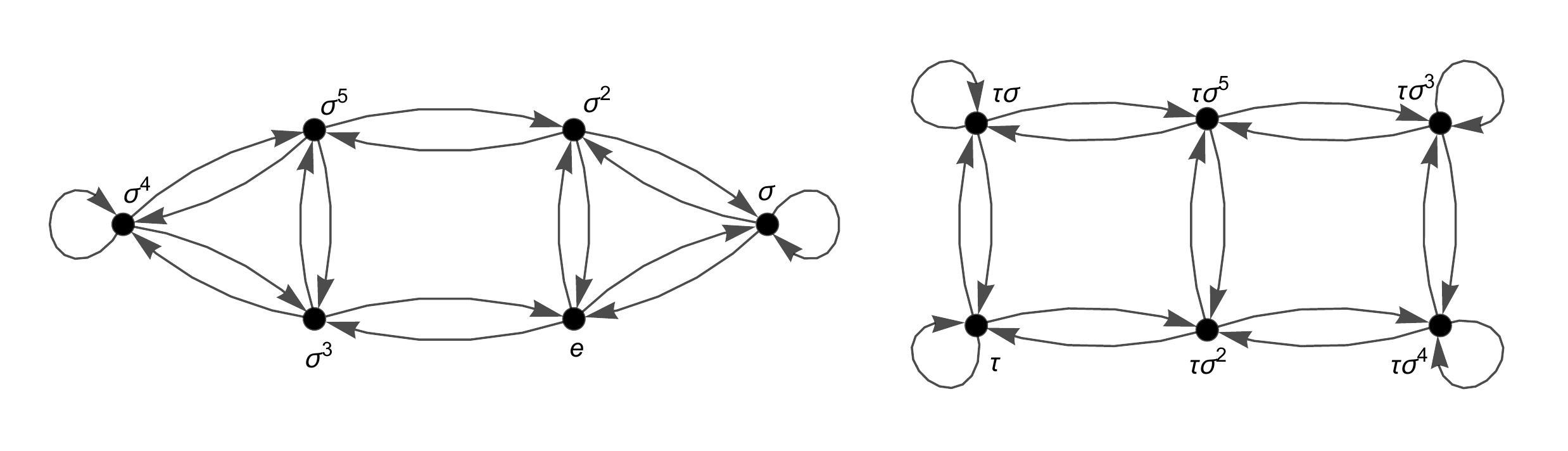}

\caption{$2\mathrm{S}(D_6;\{\tau,\tau\sigma^5\},\{\tau\sigma,\tau\sigma^2\})$}
\label{fig:Problem2D6}
\end{figure}

\begin{example}
\label{example:D6B}
Consider $2\mathrm{S}(D_{10};\{\sigma\},\{\tau,\sigma^3\})$.
It is clear that $D_{10}=\mathcal{W}(R)=\mathcal{W}(L)\mathcal{W}(R)$.  Since $e \not\sim \tau$ but $e\sim \tau^2=e$, the graph, as seen in Figure~\ref{fig:ConnCompEx}, has two strongly connected components of the same size as shown in Theorem~\ref{theorem:number_strong_components}.  Notice that $N_{D_{10}}(L)=\langle \sigma \rangle$ so $\sigma \in L\cap N_{D_{10}}(L)$ and the components are isomorphic by Theorem~\ref{theorem:number_strong_components}.
\end{example}

\begin{figure}[htbp]
\centering
\includegraphics[width=8 cm]{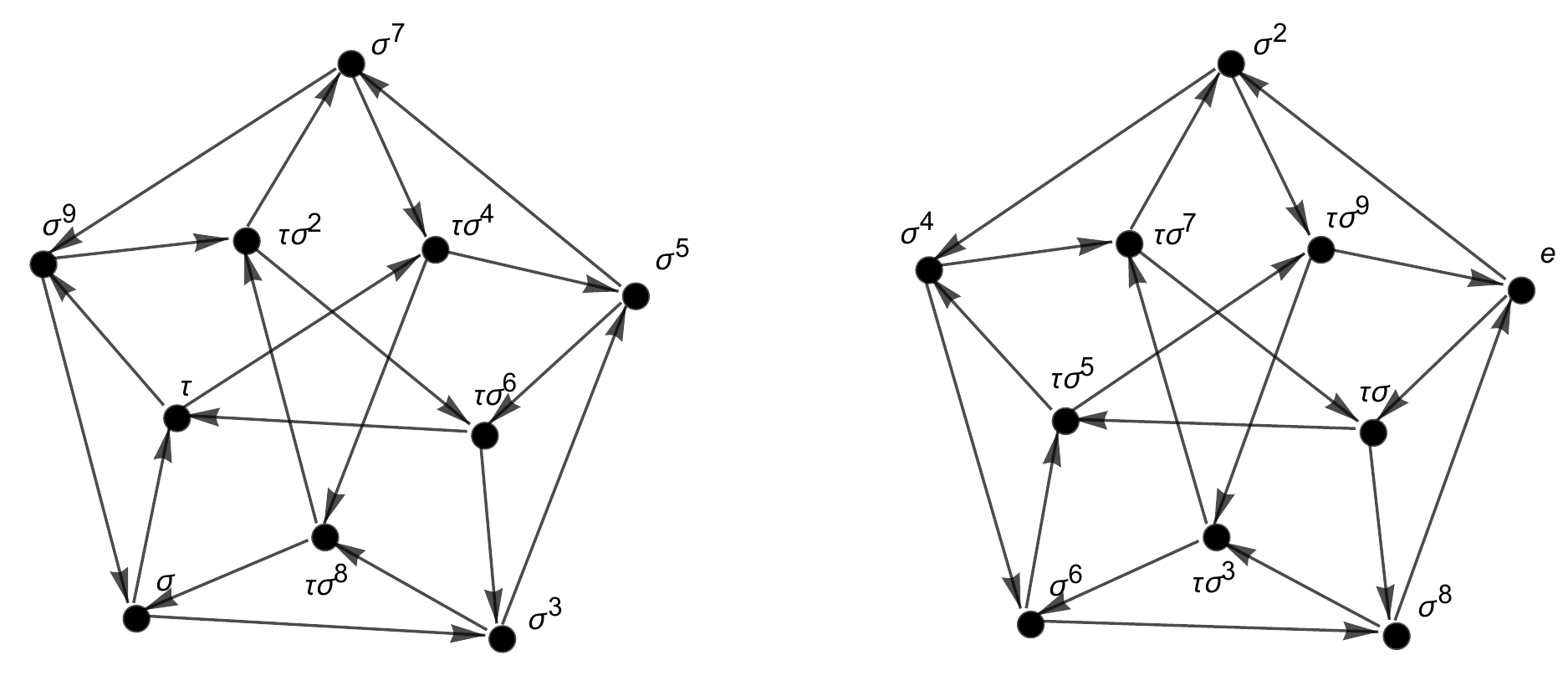}
\caption{$2\mathrm{S}(D_{10};\{\sigma\},\{\tau,\sigma^3\})$}
\label{fig:ConnCompEx}
\end{figure}

%%%%%%%%%%%%%%%%%%%%%%%%%%%%%%%%%%%%%%%%%%
%            Double Cosets               %
%%%%%%%%%%%%%%%%%%%%%%%%%%%%%%%%%%%%%%%%%%

\section{Double cosets}
\label{sec:double_cosets}

Recall that for a Cayley digraph $\text{Cay}(G, S)$ the coset $\langle S \rangle g$ is the weakly connected component of the digraph containing $g\in G$ and that if $H$ and $K$ are subgroups of a group $G$ then the double cosets $HgK$ for $g \in G$ partition $G$ into (possibly different sized) subsets. In the two-sided group digraph $2\mathrm{S}(G; L,R)$ the component containing $g\in G$ need only be contained in the double coset $\langle L \rangle g \langle R \rangle$.  

\begin{proposition}\label{prop:weak_component_coset}
The weakly or strongly connected component of $2\mathrm{S}(G; L,R)$ containing $g$ is a subset of the double coset $\langle L \rangle g \langle R \rangle$. 
\end{proposition}
\begin{proof}
Let $h$ be weakly connected to $g$; that is, $h$ is of the form $W_{\bar{L},m,a}gW_{\bar{R},a,m}$ for some $W_{\bar{L},m,a}\in \mathcal{W}(\bar{L})= \langle L\rangle$ and $W_{\bar{R},a,m}\in \mathcal{W}(\bar{R})=\langle R\rangle$. Then $h\in \langle L  \rangle g \langle R \rangle$ and the weakly or strongly connected component containing $g$ lies in $\langle L \rangle g \langle R \rangle$. 
\end{proof}

In Theorem~\ref{theorem:number_weak_components_coset}, without the assumption that $G=\mathcal{W}(\bar{L})\mathcal{W}(\bar{R})$, we count connected components within double cosets analogously to Theorem~\ref{theorem:number_weak_components}.  Connected components in a given double coset have the same size, but between different double cosets the sizes of components can differ.  This is illustrated in Figure~\ref{Fig:Dinf1} for Example~\ref{ex:finite_infinite}, Figure~\ref{fig:DoubleCosetD3xC3} for Example~\ref{example:D3xC3}, and Figure~\ref{fig:DoubleCosetA5} for Example~\ref{example:A5}.

Let $L$ and $R$ be nonempty subsets of $G$ and fix a set $S$ of double coset representatives for $\langle L \rangle$ and $\langle R \rangle$.
Each $g$ in $G$ lies in a double coset $\langle L \rangle s \langle R \rangle$ for some $s \in S$, and $s$ will play the role in $\langle L \rangle s \langle R \rangle$ that the identity element played in Sections~\ref{sec:connection_length} and~\ref{sec:connected_components}. 

\begin{lemma}
\label{lem:weak_word_manipulation_coset}
In $2\mathrm{S}(G;L,R)$ if $g=w_{\bar{L},m,a}sw_{\bar{R},n,b}$ for $s \in G$, then $g \sim l^ds$ and $g \sim sr^d$ where $l$ is any element of $L$, $r$ is any element of $R$, and $d=m+n-(a+b)$.
\end{lemma}

\begin{proof}
This proof is identical to the proof of Lemma~\ref{lem:weak_word_manipulation} with $s$ inserted between the words from $\bar{L}$ and words from $\bar{R}$.
\end{proof}

\begin{corollary}
\label{cor:weak_connection_properties_coset}
In $2\mathrm{S}(G;L,R)$ the following hold with $s \in G$.
\begin{enumerate}
\item There exist words $w_{L,m}$ and $w_{L,n}$ with $m\neq n$ such that $w_{L,m}s \sim w_{L,n}s$ if and only if there exists $w_{L,k}$ such that  $w_{L,k}s \sim s$.
One can take $k=\lvert m - n \rvert$.
\item  There exists a word $w_{L,n}$ such that $w_{L,n}s \sim s$ if and only if there exists a word $w_{L^{-1},n}$ such that $w_{L^{-1},n}s \sim s$.
\item If $g$ is in $\langle L \rangle s \langle R \rangle$ then $w_{L,k}g \sim g$ for some $w_{L,k}$ in $\mathcal{W}(L)$ if and only if $w'_{L,k}s \sim s$ for some $w'_{L,k}$ in $\mathcal{W}(L)$.
\end{enumerate}
\end{corollary}

\begin{proof}
The first two parts follow similarly to their analogues in Corollary~\ref{cor:weak_connection_properties}.

For (3), note that by symmetry it is enough to prove one direction. Let $g=w_{\bar{L},m,a}sw_{\bar{R},n,b}$ and $w_{L,k}g \sim g$. 
Rewriting $w_{L,k}g \sim g$ in terms of $s$ yields $w_{L,k}w_{\bar{L},m,a}sw_{\bar{R},n,b} \sim w_{\bar{L},m,a}sw_{\bar{R},n,b}$.
Applying Lemma~\ref{lem:weak_word_manipulation_coset} to both sides yields $l^{k+d}s\sim l^ds$.
Hence $l^ks \sim s$ by (1).
\end{proof}

By the first two parts of Corollary~\ref{cor:weak_connection_properties_coset}, if there exists a minimum length $k_s$ of a word $w$ in $L$ such that $ws \sim s$, then it is also the minimum length of such a word in $L^{-1}$, and by Corollary~\ref{cor:weak_connection_properties_coset}~(3) the minimum such length is independent of the representative of a double coset.  
Inserting $r^{k_s}r^{-k_s}$ to the right of $s$ shows that $k_s$ is also the minimum length of a word $w$ in $R$ such that $sw \sim s$, and hence, by an $R$ version of Corollary~\ref{cor:weak_connection_properties_coset}, $k_s$ is also the minimum such length of a word in $R^{-1}$.
Thus the following definition for the minimum weak connection length in $\langle L \rangle s \langle R \rangle$ is well-defined.

\begin{definition}
The \textit{minimum weak connection length in $\langle L \rangle s \langle R \rangle$} is the minimum length $k_s$ of a word $w$ purely in $L$ or $L^{-1}$ such that $ws \sim s$ in $2\mathrm{S}(G;L,R)$, or the minimum length $k_s$ of a word $w$ purely in $R$ or $R^{-1}$ such that $sw \sim s$ in $2\mathrm{S}(G;L,R)$.  Take $k_s$ to be infinite if there is no such minimum. 
Algebraically this is equivalent to the minimum length of a word $w$ purely in $L$ or $L^{-1}$ such that $s=W_{\bar{L},m,a}wsW_{\bar{R},a,m}$ for some $a, m \in \mathbb{N}$, or the minimum length of a word $w$ purely in $R$ or $R^{-1}$ such that $s=W_{\bar{L},m,a}swW_{\bar{R},a,m}$ for some $a, m \in \mathbb{N}$.
\end{definition}

\begin{theorem}
\label{theorem:number_weak_components_coset}
Let $L$ and $R$ be nonempty subsets of a group $G$. If $k_s$ is the minimum weak connection length for $\langle L \rangle s \langle R \rangle$, then the  double coset $\langle L \rangle s \langle R \rangle$ within $2\mathrm{S}(G;L,R)$ consists of exactly $k_s$ weakly connected components all of the same size. Moreover, if $L \cap N_G(L) \neq \emptyset$ or $R \cap N_G(R) \neq \emptyset$, then all components within the same double coset are isomorphic.
\end{theorem}

\begin{proof}
This follows from Lemma~\ref{lem:weak_word_manipulation_coset} and Corollary~\ref{cor:weak_connection_properties_coset} exactly as in the proof of Theorem~\ref{theorem:number_weak_components}.
\end{proof}

\begin{corollary}
In the two-sided group digraph $2\mathrm{S}(G;L,R)$ there are $\sum\limits_{s\in S} k_s$ weakly connected components, where $S$ is a set of double coset representatives for $G$ modulo $\langle L \rangle$ and $\langle R \rangle$ and $k_s$ is the minimum weak connection length for $\langle L \rangle s \langle R \rangle$.
\end{corollary}

\begin{remark}
Note that in Theorem~\ref{theorem:number_weak_components_coset} if in fact $L\cap N_G(L)=L$ and $R\cap N_G(R)=R$, then by an argument similar to that in~\cite[Theorem~1.13]{IrdamusaPraeger2016} the subgraph $\langle L \rangle s \langle R \rangle$ is vertex-transitive.
\end{remark}

\begin{figure}[htb]
\centering
\begin{tikzpicture}
\fill (-4,2) circle [radius=0.1] node [above=2pt] {$\sigma^{-2}$};
\fill (-2,2) circle [radius=0.1] node [below=2pt] {$\sigma^{-1}$};
\fill (0,2) circle [radius=0.1] node [above=2pt] {$e$};
\fill (2,2) circle [radius=0.1] node [below=2pt] {$\sigma$};
\fill (4,2) circle [radius=0.1] node [above=2pt] {$\sigma^2$};

\fill (-4,0) circle [radius=0.1] node [above=2pt] {$\tau\sigma^{-2}$};
\fill (-2,0) circle [radius=0.1] node [above=2pt] {$\tau\sigma^{-1}$};
\fill (-0,0) circle [radius=0.1] node [above=2pt] {$\tau$};
\fill (2,0) circle [radius=0.1] node [above=2pt] {$\tau\sigma$};
\fill (4,0) circle [radius=0.1] node [above=2pt] {$\tau\sigma^{2}$};

\draw [->] (-6, 1.4) to [out=0,in=210] (-4.1,1.9);
\draw [->] (-3.9, 1.9) to [out=330,in=210] (-0.1,1.9);
\draw [->] (0.1, 1.9) to [out=330,in=210] (3.9,1.9);
\draw [->] (4.1,1.9) to [out=330,in=180] (6,1.4);

\draw [->] (-5, 2.5) to [out=10,in=150] (-2.1,2.1);
\draw [->] (-1.9, 2.1) to [out=30,in=150] (1.9,2.1);
\draw [->] (2.1, 2.1) to [out=30,in=170] (5,2.5);

\draw [->] (-4.1,-0.1) to [loop below,out=225,in=315,looseness=10] (-3.9,-0.1);
\draw [->] (-2.1,-0.1) to [loop below,out=225,in=315,looseness=10] (-1.9,-0.1);
\draw [->] (-0.1,-0.1) to [loop below,out=225,in=315,looseness=10] (0.1,-0.1);
\draw [->] (1.9,-0.1) to [loop below,out=225,in=315,looseness=10] (2.1,-0.1);
\draw [->] (3.9,-0.1) to [loop below,out=225,in=315,looseness=10] (4.1,-0.1);

\draw (-5,2) circle [radius=0.02];
\draw (-5.5,2) circle [radius=0.02];
\draw (-6,2) circle [radius=0.02];

\draw (5,2) circle [radius=0.02];
\draw (5.5,2) circle [radius=0.02];
\draw (6,2) circle [radius=0.02];

\draw (-5,0) circle [radius=0.02];
\draw (-5.5,0) circle [radius=0.02];
\draw (-6,0) circle [radius=0.02];

\draw (5,0) circle [radius=0.02];
\draw (5.5,0) circle [radius=0.02];
\draw (6,0) circle [radius=0.02];
\end{tikzpicture}
\caption{$2\mathrm{S}(D_\infty;\{\sigma^a\},\{\sigma^b\})$ where $a=-1,b=1$}\label{Fig:Dinf1}
\end{figure}
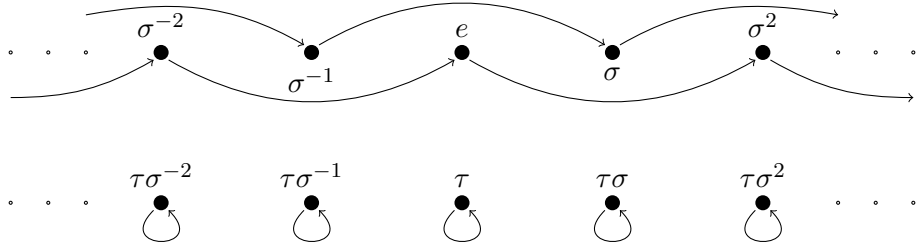

\begin{example}\label{ex:finite_infinite}
Consider $2\mathrm{S}(D_{\infty};\{\sigma^a\},\{\sigma^b\})$ with $\gcd(a,b) = 1$ and where $D_{\infty}$ is the group of isometries of $\mathbb{Z}$ with the presentation $D_{\infty}=\langle \sigma,\tau \mid \tau^2=e, \sigma\tau=\tau\sigma^{-1}\rangle$.  We think of $\sigma$ as right translation and $\tau$ as negation. Since $\langle L \rangle= \{\sigma^{an}\}$ and $\langle R \rangle=\{\sigma^{bn}\}$ with $a,b$ relatively prime, $D_{\infty}$ has two double cosets, namely $\langle L \rangle\langle R \rangle=\langle \sigma \rangle$ and $\langle L \rangle\tau\langle R \rangle=\tau\langle \sigma \rangle$.

It is easy to see that each $g\in D_\infty$ has exactly one out-neighbor and one in-neighbor (possibly the same). If $g = \sigma^n \in \langle L\rangle\langle R\rangle$, then $g$ lies on the arcs $(\sigma^n,\sigma^{n+(b-a)})$ and $(\sigma^{n-(b-a)},\sigma^n)$. If instead $g = \tau\sigma^n \in \langle L\rangle\tau\langle R\rangle$, then $g$ lies on the arcs $(\tau\sigma^n,\tau\sigma^{n+a+b})$ and $(\tau\sigma^{n-(a+b)},\tau\sigma^n)$. Therefore, the structure of the graph depends on $b-a$ and $a+b$.

If $b-a\neq 0$, then the double coset $\langle L\rangle\langle R\rangle$ consists of $|b-a|$ weakly connected components each consisting of $\sigma^n$ with $n$ fixed modulo $|b-a|$. If $b-a = 0$, then the arcs are of the form $(\sigma^n,\sigma^n)$, and the double coset consists of isolated points linked only to themselves, so has infinitely many connected components. Both of these cases illustrate the results of Theorem~\ref{theorem:number_weak_components_coset}.

The value of $a+b$ plays the same role for the structure of the double coset $\langle L\rangle\tau\langle R\rangle$.
Two example graphs are provided, Figure~\ref{Fig:Dinf1} for $a=-1,b=1$ and Figure~\ref{Fig:Dinf2} for $a=1,b=2$.

\end{example}

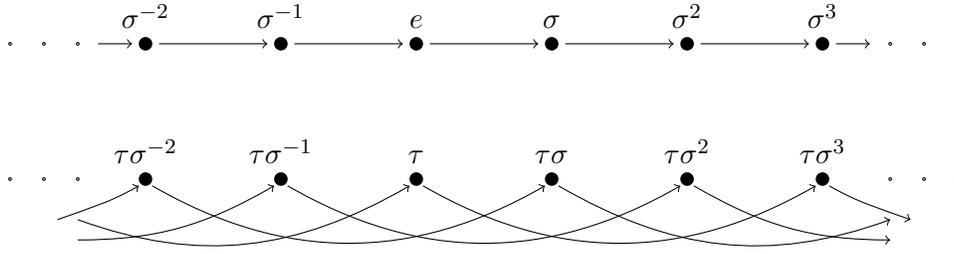
\begin{figure}[htb]
\centering
\begin{tikzpicture}[scale=0.9]
\fill (-4,2) circle [radius=0.1] node [above=2pt] {$\sigma^{-2}$};
\fill (-2,2) circle [radius=0.1] node [above=2pt] {$\sigma^{-1}$};
\fill (0,2) circle [radius=0.1] node [above=2pt] {$e$};
\fill (2,2) circle [radius=0.1] node [above=2pt] {$\sigma$};
\fill (4,2) circle [radius=0.1] node [above=2pt] {$\sigma^2$};
\fill (6,2) circle [radius=0.1] node [above=2pt] {$\sigma^3$};

\fill (-4,0) circle [radius=0.1] node [above=2pt] {$\tau\sigma^{-2}$};
\fill (-2,0) circle [radius=0.1] node [above=2pt] {$\tau\sigma^{-1}$};
\fill (-0,0) circle [radius=0.1] node [above=2pt] {$\tau$};
\fill (2,0) circle [radius=0.1] node [above=2pt] {$\tau\sigma$};
\fill (4,0) circle [radius=0.1] node [above=2pt] {$\tau\sigma^2$};
\fill (6,0) circle [radius=0.1] node [above=2pt] {$\tau\sigma^3$};

\draw [->] (-4.7,2) -- (-4.2,2);
\draw [->] (-3.8,2) -- (-2.2,2);
\draw [->] (-1.8,2) -- (-0.2,2);
\draw [->] (0.2,2) -- (1.8,2);
\draw [->] (2.2,2) -- (3.8,2);
\draw [->] (4.2,2) -- (5.8,2);
\draw [->] (6.2,2) -- (6.7,2);

\draw [->] (-3.9,-0.1) to [out=330,in=210] (1.9,-0.1);
\draw [->] (-1.9,-0.1) to [out=330,in=210] (3.9,-0.1);
\draw [->] (0.1,-0.1) to [out=330,in=210] (5.9,-0.1);

\draw [->] (-5.3,-0.6) to [out=20,in=210] (-4.1,-0.1);
\draw [->] (-5,-0.9) to [out=0,in=210] (-2.1,-0.1);
\draw [->] (-5,-0.6) to [out=340,in=210] (-0.1,-0.1);

\draw [->] (2.1,-0.1) to [out=330,in=200] (7,-0.6);
\draw [->] (4.1,-0.1) to [out=330,in=180] (7,-0.9);
\draw [->] (6.1,-0.1) to [out=330,in=160] (7.3,-0.6);

\draw (-5,2) circle [radius=0.02];
\draw (-5.5,2) circle [radius=0.02];
\draw (-6,2) circle [radius=0.02];

\draw (7,2) circle [radius=0.02];
\draw (7.5,2) circle [radius=0.02];
\draw (8,2) circle [radius=0.02];

\draw (-5,0) circle [radius=0.02];
\draw (-5.5,0) circle [radius=0.02];
\draw (-6,0) circle [radius=0.02];

\draw (7,0) circle [radius=0.02];
\draw (7.5,0) circle [radius=0.02];
\draw (8,0) circle [radius=0.02];
\end{tikzpicture}
\caption{$2\mathrm{S}(D_\infty;\{\sigma^a\},\{\sigma^b\})$ where $a=1,b=2$}\label{Fig:Dinf2}
\end{figure}

Proposition~\ref{lem:two_way_connectedness} and Corollary~\ref{cor:remark1} again yield corresponding strongly connected results.

\begin{lemma}
\label{lem:strong_word_manipulation_coset}
In $2\mathrm{S}(G;L,R)$ if $\mathcal{W}(L)$ and $\mathcal{W}(R)$ are subgroups of $G$ and $g=w_{L^{-1},a}sw_{R,n}$ for $s \in G$, then $g$ is strongly connected to $l^ds$ and to $sr^d$ where $l$ is any element of $L$, $r$ is any element of $R$, and $d=n-a$.
\end{lemma}

\begin{corollary}
\label{cor:strong_connection_properties_coset}
In $2\mathrm{S}(G;L,R)$ if $\mathcal{W}(L)$ and $\mathcal{W}(R)$ are subgroups of $G$, then the following three properties hold for any $s \in G$.
\begin{enumerate}
\item There exist words $w_{L^{-1},m}$ and $w_{L^{-1},n}$ with $m\neq n$ such that $w_{L^{-1},m}s$ is strongly connected to $w_{L^{-1},n}s$ if and only if there exists $w_{L^{-1},k}$ such that  $w_{L^{-1},k}s$ is strongly connected to $s$.
In practice, $k=\lvert m - n \rvert$.
\item There exists a word $w_{L,n}$ such that $w_{L,n}s$ is strongly connected to $s$ if and only if there exists a word $w_{L^{-1},n}$ such that $w_{L^{-1},n}s$ is strongly connected to $s$.
\item If $g$ is in $\langle L \rangle s \langle R \rangle$ then $w_{L,k}g$ is strongly connected to $g$ for some $w_{L,k}$ in $\mathcal{W}(L)$ if and only if $w'_{L,k}s$ is strongly connected to $s$ for some $w'_{L,k}$ in $\mathcal{W}(L)$.
\end{enumerate}
\end{corollary}

\begin{definition}
Assuming that $\mathcal{W}(L)$ and $\mathcal{W}(R)$ are subgroups of $G$, the \textit{minimum strong connection length in $\langle L \rangle s \langle R \rangle$} is the minimum length $k_s$ of a word $w$ purely in $L$ or $L^{-1}$ such that $ws$ is strongly connected to $s$ in $2\mathrm{S}(G;L,R)$, or the minimum length $k_s$ of a word $w$ purely in $R$ or $R^{-1}$ such that $sw$ is strongly connected to $s$ in the two-sided group digraph $2\mathrm{S}(G;L,R)$.  Take $k_s$ to be infinite if there is no such minimum. 
Algebraically this is equivalent to the minimum length of a word $v$ purely in $L$ or $L^{-1}$ such that $s=w_{L^{-1},n}vsw_{R,n}$ for some $n \in \mathbb{N}$, or the minimum length of a word $v$ purely in $R$ or $R^{-1}$ such that $s=w_{L^{-1},n}svw_{R,n}$ for some $n \in \mathbb{N}$.
\end{definition}

\begin{theorem}
\label{theorem:number_strong_components_coset}
Let $\mathcal{W}(L)$ and $\mathcal{W}(R)$ be subgroups of $G$. If $k_s$ is the minimum strong connection length for $\langle L \rangle s \langle R \rangle$ in $2\mathrm{S}(G;L,R)$, then the  double coset $\langle L \rangle s \langle R \rangle$ consists of exactly $k_s$ strongly connected components all of the same size. Moreover, if $L \cap N_G(L) \neq \emptyset$ or $R \cap N_G(R) \neq \emptyset$, then all components within the same double coset are isomorphic.
\end{theorem}

\begin{problem}
Theorems~\ref{theorem:number_weak_components},~\ref{theorem:number_strong_components},~\ref{theorem:number_weak_components_coset}, and~\ref{theorem:number_strong_components_coset}
provide sufficient conditions for connected components to be isomorphic.  Find necessary and sufficient conditions for this to occur.
\end{problem}

\begin{corollary}
Let $\mathcal{W}(L)$ and $\mathcal{W}(R)$ be subgroups of $G$. The two-sided group digraph $2\mathrm{S}(G;L,R)$ consists of $\sum\limits_{s\in S} k_s$ strongly connected components, where $S$ is a set of double coset representatives for $G$ modulo $\langle L \rangle$ and $\langle R \rangle$ and $k_s$ is the minimum strong connection length for $\langle L \rangle s \langle R \rangle$.
\end{corollary}

\begin{example}
\label{example:D3xC3}
The digraph $2\mathrm{S}(D_3\times C_3;\{(\tau\sigma^2,g^2)\},\{(e,g^2),(\tau,g^2)\})$ shown in Figure~\ref{fig:DoubleCosetD3xC3} is an example of Theorem~\ref{theorem:number_strong_components_coset}.
The two double cosets in $G=D_3\times C_3$ are $\langle L \rangle \langle R \rangle$ and $\langle L \rangle (\sigma^2,e) \langle R \rangle$, both of which have minimum strong connection length of three.  Since $L$ consists of a single element, $L\cap N_G(L)\neq \emptyset$ and all components within each double coset are isomorphic.
\end{example}

\begin{figure}[htb]
\centering
\includegraphics[width=10 cm]{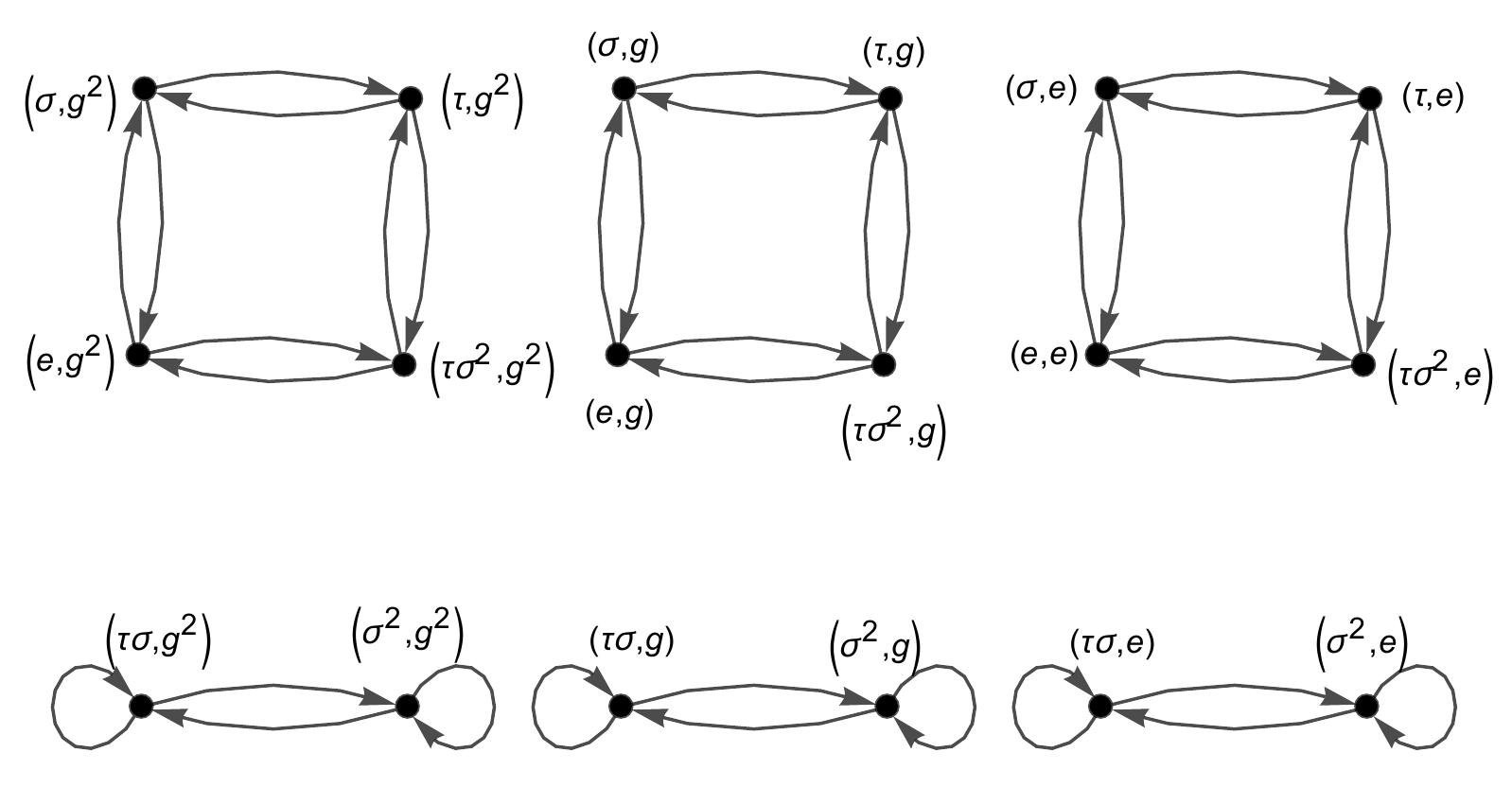}
\caption{$2\mathrm{S}(D_3\times C_3;\{(\tau\sigma^2,g^2)\},\{(e,g^2),(\tau,g^2)\})$}
\label{fig:DoubleCosetD3xC3}
\end{figure}

\begin{example}
\label{example:A5}
Another example is provided by $2\mathrm{S}(A_5;\{(235)\},\{(243),(254)\})$, shown in Figure~\ref{fig:DoubleCosetA5}.
There are three double cosets in $A_5$ modulo $\langle L \rangle$ and $\langle R \rangle$, whose representatives are the identity, $(123)$, and $(145)$. The minimum strong connection length is three in the first two double cosets and one in the third.  The connected components of $\langle L \rangle \langle R \rangle$ have four vertices.
All the connected components in the other two double cosets contain $12$ vertices and are isomorphic.  
That the components within $\langle L \rangle (123) \langle R \rangle$ are isomorphic follows from the fact that $L$ consists of a single element so $L\cap N_G(L)\neq \emptyset$.
\end{example}

\begin{figure}[htb]
\centering
\includegraphics[width=10cm]{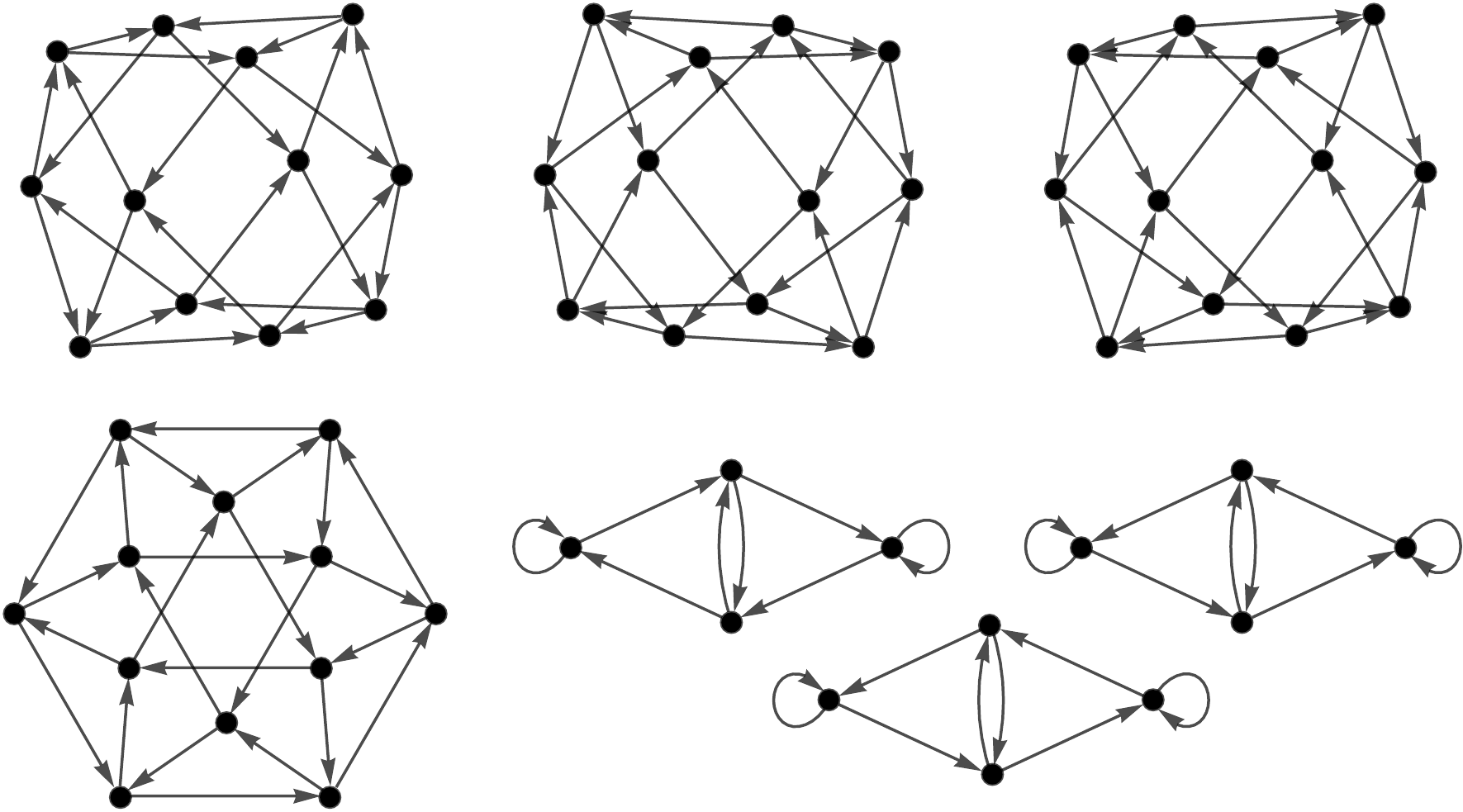}
\caption{$2\mathrm{S}(A_5;\{(235)\},\{(243),(254)\})$}
\label{fig:DoubleCosetA5}
\end{figure}

%%%%%%%%%%%%%%%%%%%%%%%%%%%%%%%%%%%%%%%%%%
%           Orbit Counting               %
%%%%%%%%%%%%%%%%%%%%%%%%%%%%%%%%%%%%%%%%%%

\section{Orbit counting}
\label{sec:orbitcounting}
Another way to count strongly connected components is to use group actions.
We briefly review necessary background material.

A group $G$ \textit{acts} (on the right) on a set $X$ if there exists a function $\alpha: X\times G \rightarrow X$, where $(x, g)\mapsto x.g$ 
such that $x.e=x$ and for all $g_1, g_2\in G$ and all $x\in X$, $x\cdot(g_1g_2)=(x \cdot g_1).g_2.$ If $G$ acts on a set $X$, then for any $x\in X$, the set $x.G=\left\{x.g\mid g\in G\right\}$ is the \textit{orbit} of $x$ under $G$. It can be shown that $X$ is the disjoint union of its orbits.
If $G$ acts on $X$, the \textit{stabilizer} of $x\in X$ is the subgroup $G_x=\left\{ g\mid x.g=x\right\}$ of $G$ and the set fixed by $g \in G$ is $X^g=\{x \mid x.g=x\}$. 
The following well-known results are used to prove Theorem~\ref{theorem:orbit_counting}.

\begin{lemma}\label{lem:grp_size_prod_orbit_stab}
Suppose that a group $G$ acts on a set $X$. If $x\in X$, then the mapping $\phi: G_x\backslash G \to x.G$ defined by $\phi(G_x g)=x.g$ is well-defined and bijective. Thus, $|G|=|x.G||G_x|.$ 
\end{lemma}

\begin{lemma}\label{lem:stabilizer_conjugate}
Suppose that a group $G$ acts on a set $X$.
\begin{enumerate}
\item  If $x\in X$ and $g\in G$, the stabilizer of $x.g$ is $G_{x.g}=g^{-1}G_xg.$ 
\item  If $x$ and $y$ are in the same orbit under $G$,  then $|G_x|=|G_y|.$ 
\end{enumerate}
\end{lemma}

\begin{theorem}\label{theorem:orbit_counting}
Suppose that a group $G$ acts on a set $X$. The number $N$ of distinct orbits of $G$ on $X$ satisfies
\[N\cdot |G|=\sum_{g\in G}|X^g|.
\]
\end{theorem}
\begin{proof}
The case where $X$ or $G$ is infinite is trivial so let $X$ and $G$ be finite.
Consider the set  $Y=\left\{(x, g)\mid g\in G, x\in X, x.g=x\right\}\subset X\times G$. 
We may count elements of $Y$ as $|Y|=\sum\limits_{g\in G}|X^g|=\sum\limits_{x\in X}|G_x|$.
Alternatively, consider representatives $x_1, x_2, x_3, \cdots , x_N$ from each orbit of $X$. If $x$ is in the same orbit as $x_i$, then $x.G=x_i.G$ and hence, by Lemma~\ref{lem:stabilizer_conjugate}, $|G_x|=|G_{x_i}|$.
We therefore have, by Lemma~\ref{lem:grp_size_prod_orbit_stab},
\[
\sum_{g\in G}|X^g|=\sum_{i=1}^{N}\sum_{x\in x_i.G}|G_x|=\sum_{i=1}^{N}|x_i.G||G_{x_i}|=\sum_{i=1}^{N}|G|=N\cdot |G|.
\]\end{proof}

We apply this result to $2\mathrm{S}(G;L,R)$.
Define \[U=\left\{(w_{L,n},w_{R,n}) \mid w_{L,n}\in \mathcal{W}(L), w_{R,n}\in \mathcal{W}(R)\right\} \subseteq G \times G.\]
We show that if $\mathcal{W}(L)$ and $\mathcal{W}(R)$ are subgroups of $G$ then $U$ is a subgroup of $G\times G$. 
The set $U$ is clearly closed under multiplication.
The fact that $U$ is closed under inverses follows from the proof of Proposition~\ref{lem:two_way_connectedness}.  Since $U$ is not empty it contains an identity and $U$ is a group under composition. 

The action of $U$ on $G$ is induced by the standard action of $G \times G$ on $G$ by $g \cdot (g_1,g_2)=g_1^{-1} g g_2$, i.e.,   $g\cdot (w_{L,n},w_{R,n})=w_{L,n}^{-1}gw_{R,n}$.
One can check that this is in fact a right action.  
For each element $g$ in $G$, the orbit $g\cdot U$ is the strongly connected component of $2\mathrm{S}(G; L,R)$ containing $g$. 

\begin{corollary}
\label{cor:groupaction}
Let $2\mathrm{S}(G; L,R)$ be a two-sided group digraph where $\mathcal{W}(L)$ and $\mathcal{W}(R)$ are subgroups of $G$ and with the group $U=\left\{(w_{L,n},w_{R,n}) \mid w_{L,n}\in \mathcal{W}(L), w_{R,n}\in \mathcal{W}(R)\right\}$ acting on $G$ as defined above.
The number $N$ of strongly connected components in $2\mathrm{S}(G; L,R)$ satisfies $N \cdot|U|=\sum\limits_{u \in U} |G^{u}|$.
\end{corollary}

\begin{example}
Let $2\mathrm{S}(G; L,R)$ be a connected digraph and let $H_N$ be any group of order $N$.  Then $2\mathrm{S}(G \times H_N; L \times \{e\}, R \times \{e\})$ has $N$ connected components.  This shows that the number $N$ of connected components may be arbitrarily large.
\end{example}

\begin{problem}
\label{problem3}
For a given group $G$, how many connected components can $2\mathrm{S}(G; L,R)$ have?  \end{problem}

Note that if $G=\langle L\rangle \langle R\rangle$ then by Theorem~\ref{theorem:number_weak_components} or Theorem~\ref{theorem:number_strong_components} the number of connected components will divide $|G|$, but Example~\ref{example:A5} shows this need not hold in general.

Based on the group action perspective and our observation about the connection between two-sided group digraphs and biquotients we pose a question motivated by a common construction in the biquotient setting.  We first define a generalization of $2\mathrm{S}(G; L,R)$.

\begin{definition}
Let $G$ be a group and $U$ be a non-empty subset of $G\times G$.  Define the digraph $2\mathrm{S}(G ; U)$ to have vertex set $G$ and a directed arc $(g,h)$ from $g$ to $h$ if and only if $h = u_l^{-1} g u_r$ for some $(u_l,u_r)\in U$.
\end{definition}

\begin{remark}
Observe that if $U = L\times R$, then $2\mathrm{S}(G; U) = 2\mathrm{S}(G; L, R)$.
\end{remark}

Motivated by the biquotient literature, we note a correspondence between the digraphs $2\mathrm{S}(G;U)$ and $2\mathrm{S}(G\times G ; \Delta G, U)$ where $\Delta G=\{(g,g)\mid g \in G\}$ is the diagonal of $G\times G$.  This correspondence is given by the map $\phi:G\times G \to G$, $\phi(g_1,g_2) = g_1^{-1}g_2$. Direct computation shows that the map $\phi$ takes arcs of $2\mathrm{S}(G\times G ; \Delta G, U)$ to arcs of $2\mathrm{S}(G;U)$. Additionally, it produces a bijection between the connected components of $2\mathrm{S}(G\times G ; \Delta G, U)$ and the connected components of $2\mathrm{S}(G;U)$. This allows the number of connected components of $2\mathrm{S}(G;U)$ to be counted using our preceding results, especially when one notes that by Theorem~\ref{theorem:number_weak_components_coset} and Theorem~\ref{theorem:number_strong_components_coset} the connected components of  $2\mathrm{S}(G\times G ; \Delta G, U)$ are precisely the double cosets.

\begin{problem}
Under what conditions on $U$ do the connected components of $2\mathrm{S}(G;U)$ have the same size? Under what conditions are they isomorphic?
\end{problem}

%%%%%%%%%%%%%%%%%%%%%%%%%%%%%%%%%%%%%%%%%%
%                Reduction               %
%%%%%%%%%%%%%%%%%%%%%%%%%%%%%%%%%%%%%%%%%%

\section{Reduction results}
\label{sec:reduction}

In this section we prove Proposition~\ref{prop:subgroup_connected}, which relates connectedness of a two-sided digraph for a semi-direct product group to connectedness properties for the factors, and pose a final general problem.
Remark~\ref{rem:retract} will be useful in the proof of Proposition~\ref{prop:subgroup_connected}.

\begin{remark}
\label{rem:retract}
A group $G$ is said to be a \textit{semi-direct product} of its subgroups $H$ and $K$, written $G=H \rtimes K$, if $H$ is a normal subgroup of $G$, $G=HK$, and $H \cap K =\{e\}$.
A subgroup $K$ of a group $G$ is a \textit{retract} of $G$ if there exists a homomorphism $\phi:G \to G$ such that $\phi(g) \in K$ for all $ g \in G$ and $\phi(k)=k$ for all $k \in K$.

If $K$ is a retract of a group $G$ with retraction map $\phi$, then it is easy to verify that $G=H \rtimes K$ for $H=\ker \phi$.
Conversely, if $G=H \rtimes K$ then the map $\phi$ defined by $\phi(hk)=k$ is well-defined because $H \cap K =\{e\}$ and is a group homomorphism because $H\trianglelefteq G$.
Hence $G=H \rtimes K$ if and only if $K$ is a retract of $G$ with retraction $\phi$ and $H=\ker \phi$.  Denote $\phi(L)$ by $L^\phi$.
\end{remark}

\begin{proposition} \label{prop:subgroup_connected}
Let $K$ be a retract of a group $G$ under retraction $\phi$.
Then $2\mathrm{S}(G;L,R)$ is weakly connected if and only if $2\mathrm{S}(K;L^\phi,R^\phi)$ is weakly connected and $\ker \phi$ is weakly connected within $2\mathrm{S}(G;L,R)$.
\end{proposition}

\begin{proof}
Assume that $2\mathrm{S}(G;L,R)$ is weakly connected.  
Then certainly $H=\ker\phi$ is weakly connected within $2\mathrm{S}(G;L,R)$.
Observe that $2\mathrm{S}(K;L^\phi,R^\phi)$ is also weakly connected because the retraction $\phi:G \to K$ sends the arc $(g,l^{-1} g r)$ in $2\mathrm{S}(G;L,R)$ to the arc 
\[(\phi(g),\phi(l^{-1} g r))=(\phi(g), \phi(l)^{-1}\phi(g)\phi(r))\]
in $2\mathrm{S}(K;L^\phi,R^\phi)$, i.e., $\phi$ induces a retraction from $2\mathrm{S}(G,L,R)$ to $2\mathrm{S}(K;L^\phi,R^\phi)$.

Conversely, assume that $2\mathrm{S}(K;L^\phi,R^\phi)$ is weakly connected and $H=\ker\phi$ is weakly connected within $2\mathrm{S}(G;L,R)$.
We show that for every $g \in G$ there is a path in $2\mathrm{S}(G;L,R)$ from the identity to $g$.
Write $g=hk$ for $h \in H$ and $k \in K$.  Using that there is a path from $e$ to $k$ in $2\mathrm{S}(K,L^\phi,R^\phi)$, write $k=W_{\bar{L}^\phi,m,a}W_{\bar{R}^\phi,a,m}$ and then
\[g=hw_{\bar{L}^\phi,m,a}w_{\bar{R}^\phi,a,m}.\]
For each factor $k_i \in \bar{R}^\phi$ in $w_{\bar{R}^\phi,a,m}$, find $h_i \in H$ so that $h_ik_i \in \bar{R}$ and insert $h_i^{-1}h_i$ before $k_i$ in $w_{\bar{R}^\phi,a,m}$.
Insert similarly appropriate expressions for the identity before each factor from $\bar{L}^\phi$ in $w_{\bar{L}^\phi,m,a}$.
Then use $H \trianglelefteq G$ to rewrite $g$ as
\[g=W_{\bar{L},m,a}h'W_{\bar{R},a,m},\]
where $h' \in H$, exhibiting a path from $h'$ to $g$.
Since there is a path from $e$ to $h'$ in $2\mathrm{S}(G;L,R)$, there is also a path from $e$ to $g$ in $2\mathrm{S}(G;L,R)$.
This proves $2\mathrm{S}(G;L,R)$ is weakly connected.
\end{proof}

\begin{problem}
Develop analogues of earlier results about numbers of connected components and isomorphisms between them in the setting of semi-direct products. 
\end{problem}

\begin{example}
Consider the digraph $2\mathrm{S}(D_6;\{\sigma\},\{\sigma^2,\tau\})$, where $D_6=\langle\sigma\rangle\rtimes\langle\tau\rangle$. Given $\sigma^n\in D_6$, the arc $(\sigma^n,\sigma^{-1}\sigma^n\sigma^2)=(\sigma^n,\sigma^{n+1})$ shows that $\langle\sigma\rangle$ is weakly connected in $2\mathrm{S}(D_6;\{\sigma\},\{\sigma^2,\tau\})$. Furthermore, $2\mathrm{S}(\langle\tau\rangle;L^\phi,R^\phi)=2\mathrm{S}(\langle\tau\rangle;\{e\},\{e,\tau\})$ is connected since the graph consists of two vertices $e$ and $\tau$ with arcs between $e$ and $\tau$ and loops at each. Therefore by Proposition~\ref{prop:subgroup_connected}, $2\mathrm{S}(D_6;\{\sigma\},\{\sigma^2,\tau\})$ is weakly connected.
\end{example}

\begin{example}
The two-sided group digraph $2\mathrm{S}(D_6;\{\tau,\tau\sigma^5\},\{\tau\sigma,\tau\sigma^2\})$ in Example~\ref{example:D6B} is disconnected.
Here $2\mathrm{S}(K;L^\phi,R^\phi)$ consists of isolated vertices $e$ and $\tau$ with a loop at each and $H$ is weakly connected within 
$2\mathrm{S}(D_6;\{\tau,\tau\sigma^5\},\{\tau\sigma,\tau\sigma^2\})$.
\end{example}

Using an argument similar to the one in the proof of Proposition~\ref{prop:subgroup_connected}, one can prove the following.

\begin{corollary}\label{cor:normal_retract}
Given a group $G$ and a normal subgroup $N$ let $\phi:G\to G/N$ be the canonical projection. Then $2\mathrm{S}(G;L,R)$ is weakly connected if and only if $2\mathrm{S}(G/N;L^\phi,R^\phi)$ is weakly connected and $N$ is weakly connected within $2\mathrm{S}(G;L,R)$.
\end{corollary}

In both  Proposition~\ref{prop:subgroup_connected} and Corollary~\ref{cor:normal_retract} under the further assumption that $\mathcal{W}(L)$ and $\mathcal{W}(R)$ are subgroups of $G$ similar conclusions hold for strong connectedness.

\subsection*{Acknowledgment}
We thank Briana Foster-Greenwood for helpful discussions, particularly during the initial phases of this project, and Yu Chen for the observation behind Corollary~\ref{cor:groupaction}.  We thank the referee for the encouragement to add further problems and for particular suggestions leading to Problems~\ref{problem1} and~\ref{problem3}.

\bibliographystyle{alpha}
\bibliography{2sgdg}

\begin{thebibliography}{MSZS92}

\bibitem[ABR90]{ABR1990}
Fred Annexstein, Marc Baumslag, and Arnold~L. Rosenberg.
\newblock Group action graphs and parallel architectures.
\newblock {\em SIAM J. Comput.}, 19(3):544--569, 1990.

\bibitem[Ani12]{Anil2012}
K.~V. Anil.
\newblock Generalized cayley digraphs.
\newblock {\em Pure Math. Sci.}, 1(1):1--12, 2012.

\bibitem[DeV11]{DeVito-thesis}
J.~DeVito.
\newblock {\em The Classification of Simply Connected Biquotients of Dimention
  at most 7 and 3 New Examples of Almost Positively Curved Manifolds}.
\newblock PhD thesis, Univ. of Penn., 2011.

\bibitem[Esc82]{Eschenburg-Paper}
J.-H. Eschenburg.
\newblock New examples of manifolds with strictly positive curvature.
\newblock {\em Invent. Math.}, 66(3):469--480, 1982.

\bibitem[Esc84]{Eschenburg-Habilitation}
J.-H. Eschenberg.
\newblock {\em Freie Isometrische Aktionen auf Kompakten Lie-Gruppen mit
  Positiv Gekr\"ummten Orbitr\"aumen}.
\newblock Habilitation, Univ. M\"unster, 1984.

\bibitem[Gau97]{Gauyacq1997}
Ginette Gauyacq.
\newblock On quasi-{C}ayley graphs.
\newblock {\em Discrete Appl. Math.}, 77(1):43--58, 1997.

\bibitem[GM74]{Gromoll-Meyer}
D.~Gromoll and W.~Meyer.
\newblock An exotic sphere with nonnegative sectional curvature.
\newblock {\em Ann. of Math.}, 100(2):401--406, 1974.

\bibitem[IP16]{IrdamusaPraeger2016}
M.~N. Iradmusa and C.~E. Praeger.
\newblock Two-sided group digraphs and graphs.
\newblock {\em J.~Graph Theory}, 82(3):279--295, 2016.

\bibitem[KP03]{KP2003}
Andrei~V. Kelarev and Cheryl~E. Praeger.
\newblock On transitive {C}ayley graphs of groups and semigroups.
\newblock {\em European J. Combin.}, 24(1):59--72, 2003.

\bibitem[MSZS92]{MSS1992}
Dragan Maru{\v{s}}i{\v{c}}, Raffaele Scapellato, and Norma Zagaglia~Salvi.
\newblock Generalized {C}ayley graphs.
\newblock {\em Discrete Math.}, 102(3):279--285, 1992.

\end{thebibliography}

\end{document}